\documentclass[]{sigplanconf}

\usepackage{amsmath,amssymb,latexsym,amsthm,url,multicol,bbm}
\usepackage{verbatim}
\usepackage{enumitem}
\setlist[enumerate]{itemsep=0mm}

\makeatletter
\newtheorem*{rep@theorem}{\rep@title}
\newcommand{\newreptheorem}[2]{%
\newenvironment{rep#1}[1]{%
 \def\rep@title{#2 \ref{##1}}%
 \begin{rep@theorem}}%
 {\end{rep@theorem}}}
\makeatother

\newtheorem{theorem}{Theorem}[section]
\newreptheorem{theorem}{Theorem}

\newtheorem{proposition}[theorem]{Proposition}
\newreptheorem{proposition}{Proposition}
\newtheorem{lemma}[theorem]{Lemma}
\newreptheorem{lemma}{Lemma}

\newtheorem{defi}[theorem]{Definition}
\newtheorem{conv}[theorem]{Convention}
\newtheorem{rema}[theorem]{Remark}
\newreptheorem{rema}{Remark}
\newtheorem{remaconv}[theorem]{Remark and Convention}
\newtheorem{exam}[theorem]{Example}

\newenvironment{definition}{\begin{defi}\rm}{\hfill $\lhd$\end{defi}}
\newenvironment{convention}{\begin{conv}\rm}{\end{conv}}
\newenvironment{remark}{\begin{rema}\rm}{\hfill $\lhd$\end{rema}}
\newenvironment{remarkconv}{\begin{remaconv}\rm}{\hfill $\lhd$\end{remaconv}}
\newenvironment{example}{\begin{exam}\rm}{\hfill $\lhd$\end{exam}}

\renewenvironment{proof}{\begin{trivlist}\item[]{\bf
Proof.}}{\hfill {\sc qed}\end{trivlist}}

\newtheorem{claim2}{\sc Claim}
\newenvironment{claim}{\begin{claim2}\rm}{\end{claim2}\rm}
\newenvironment{claimfirst}{\setcounter{claim2}{0}
               \begin{claim2}\rm}{\end{claim2}\rm}

\newenvironment{pfclaim}{\begin{trivlist}\item[]{\sc Proof of
Claim}}{\hfill {\mbox{$\blacktriangleleft$}}\end{trivlist}}

\renewcommand{\b}[1]{\mathbb{#1}}
\newcommand{\heading}[1]{\begin{center}-- {\bf #1} --\end{center}}
\newcommand{\intpart}[1]{\noindent {\bf #1.} }

\newcommand{\proves}{\vdash}
\newcommand{\forces}{\Vdash}

\newcommand{\X}{\mathrm{X}}
\newcommand{\I}{\mathrm{I}}
\newcommand{\EU}{\mathrm{EU}}
\newcommand{\EG}{\mathrm{EG}}
\newcommand{\AF}{\mathrm{AF}}
\newcommand{\AR}{\mathrm{AR}}

\newcommand{\diam}{\dia}
\renewcommand{\a}{\mathrm{a}}
\newcommand{\e}{\mathrm{e}}
\newcommand{\f}{\mathrm{f}}

\newcommand{\chip}{\widetilde{\chi}} 
\newcommand{\FO}{\text{FO}}
\newcommand{\MSO}{\text{MSO}}

\newcommand{\LTL}{\text{LTL}}
\newcommand{\acc}{\mathsf{acc}}
\newcommand{\p}{\overline{p}}
\newcommand{\q}{\overline{q}}
\newcommand{\x}{\overline{x}}
\newcommand{\y}{\overline{y}}

\newcommand{\oa}{\overline{a}}
\newcommand{\ob}{\overline{b}}
\newcommand{\CTL}{\ensuremath{\text{\textup{CTL}}}}        
\newcommand{\CTLf}{\CTL^{f}}               
\newcommand{\CTLfp}{\CTL^{f}_{I}}               
\newcommand{\CTLfpb}{\CTL^{f}_{I,0,1}}               
\newcommand{\dia}{\Diamond}

\newcommand{\hs}{\heartsuit}
\renewcommand{\phi}{\varphi} 
\newcommand{\isbnf}{\mathrel{::=}}
\newcommand{\divbnf}{\mid}

\begin{document}

\setlength{\pdfpageheight}{\paperheight}
\setlength{\pdfpagewidth}{\paperwidth}

\conferenceinfo{CONF 'yy}{Month d--d, 20yy, City, ST, Country}
\copyrightyear{20yy}
\copyrightdata{978-1-nnnn-nnnn-n/yy/mm}
\copyrightdoi{nnnnnnn.nnnnnnn}



\title{Monadic second order logic as the \\ model companion of temporal logic}

\authorinfo{Silvio Ghilardi}
           {Universit\`a degli Studi di Milano}
           {silvio.ghilardi@unimi.it}
\authorinfo{Sam van Gool}
           {City College of New York}
           {samvangool@me.com}

\maketitle

\begin{abstract}
The main focus of this paper is on bisimulation-invariant MSO, and more particularly on giving a novel model-theoretic approach to it.
In model theory, a model companion of a theory is a first-order description of the class of models in which all potentially solvable systems of equations and non-equations have solutions. We show that bisimulation-invariant {\sc MSO} on trees gives the model companion for %
a new temporal logic, ``fair {\sc CTL}'', 
an enrichment of {\sc CTL} with local fairness constraints. To achieve this, we give a completeness proof for the logic fair CTL which combines tableaux and Stone duality, and a fair {\sc CTL} encoding of the automata for the modal $\mu$-calculus. Moreover, we also show that {\sc MSO} on binary trees is the model companion of binary deterministic fair {\sc CTL}.
\end{abstract}



\section{Introduction}\label{s:introduction}
Our main aim in this paper is to introduce the mathematical concept of model-completeness into the study of MSO, which is fundamental to computer science, and to connect it to temporal tree logic. 
In a slogan, our main thesis is that monadic second order logic `is' the {\it model companion} of temporal logic.

While model-completeness, as many topics in computer science, has its origins in mathematical logic, since the early 2000's this concept has become relevant for computer science. The most important application of model-completeness concerns automated reasoning in first-order logic, in particular, for combining first-order decision procedures in the case of non-disjoint signatures~\cite{JAR}. We plan further applications to conservativity of ontology extensions~\cite{KR}. 

In this introductory section, we give some background and motivation for model companions and we then describe our main contributions in this paper.\\

\intpart{Solving equations and model companions}
Finding solutions to equations is a challenge at the heart of both mathematics and computer science. Model-theoretic algebra, originating with the ground-breaking work of \cite{Rob1951, Robinson}, cast this problem of solving equations in a logical form, and used this setting to solve algebraic problems via model theory. 

The central notion is that of an {\it existentially closed model}, which we explain now. Call a quantifier-free formula\footnote{
	In some contexts, including the ones in this paper, quantifier-free formulas reduce to systems of equations; the notion is then also called {\it algebraically closed}.
} with parameters in a model $M$ {\it solvable} if there is an extension $M'$ of $M$ where the formula is satisfied. A model $M$ is {\it existentially closed} if any solvable quantifier-free formula already has a solution in $M$ itself. For example, the field of real numbers is not existentially closed, but the field of complex numbers is.

Although this definition is formally clear, it has a main drawback: it is not first-order definable in general. However, in fortunate and important cases, the class of existentially closed models of $T$ are exactly the models of another first-order theory $T^*$. In this case, the theory $T^*$ can be characterized abstractly as the {\it model companion} of $T$ (cf. Definition~\ref{d:modelcompanion}).

Thus, the model companion of a theory identifies the class of those models where {\it all satisfiable existential statements can be satisfied}. For example, the theory of algebraically closed fields is the model companion of the theory of fields, and dense linear orders without endpoints give the model companion of linear orders.\\

\intpart{Logic and algebra} The well-known Lindenbaum-Tarski construction shows that classical propositional logic corresponds to the class of Boolean algebras. In the same way, intuitionistic logic corresponds to Heyting algebras, and many modal and temporal logics correspond to classes of Boolean algebras enriched with operators, cf., e.g., \cite{rasiowa-sikorski}.
In this context, an existentially closed algebra corresponds to a propositional theory where `all solvable logic equations actually have a solution'. But do model companions exist in algebraic logic?\\

\intpart{Model companions in algebraic logic} Boolean algebras have a model companion: the theory of atomless Boolean algebras. The first results on model companions in modal logic were negative: the class of existentially closed modal algebras for the basic modal logic $\mathbf{K}$ is not elementary~\cite{Lipparini}. This initially discouraged further investigations in this direction, until the surprising result~\cite{Pitts} that second order intuitionistic propositional calculus can be interpreted in ordinary propositional intuitionistic calculus. As pointed out in~\cite{GZ_APAL}, this result precisely says that the theory of Heyting algebras has a model companion. 
We refer to the book~\cite{GZ} for a more complete picture of the subsequent literature on model companions for modal and intuitionistic logics.

One way to interpret the already cited result that $\mathbf{K}$ does not have a model companion is that the basic modal language is too poor. In order to obtain a first-order setting where `all solvable equations can be solved', we need to enrich the language; to this end, we will add certain \emph{fixpoints} to the modal language.\\

\intpart{Infinite words and {\sc LTL}} 
As a first step, in our forthcoming paper \cite{GvG}, we showed that a class of rooted algebras corresponding to linear temporal logic $\LTL$~\cite{Pnueli} has a model companion. Moreover, this model companion is the theory axiomatized by the sentences which are true in the special $\LTL$-algebra given by the power set of the natural numbers. In more intuitive terms, $\LTL$ has a model companion, and `it is monadic second order logic', viewed here as the first-order theory of a powerset Boolean algebra with operators. An important ingredient for the proof of this result is the fact \cite{Buc1962} that the B\"uchi acceptance condition for automata on infinite words can be converted into an existential formula in linear temporal logic.\\

\intpart{Main contributions of this paper: infinite trees and fair {\sc CTL}} 
In this paper, we exhibit model companions for the much more challenging `branching time' case. 

The most obvious candidate replacement for $\LTL$ is computational tree logic $\CTL$~\cite{Emerson}. This logic, however, turns out not to be sufficiently expressive.
%
%
%
The {\bf first contribution} of this paper is therefore the design of an extension of $\CTL$
(cf. Section~\ref{s:fairCTL}). The choice for this extension, that we call `fair CTL', is dictated by the fact that we want a logic such that bisimulation-invariant MSO is its model companion. For this purpose we need a temporal logic that can express, in a quantifier-free way, the concept of ``successful run'' of a tree automaton. The logic fair CTL seems a `minimal' extension of CTL which is sufficient for this purpose. 

The main change in moving from $\CTL$ to fair $\CTL$ is that we replace the unary $\CTL$ operator $\EG$ by a binary operator. A formula $\EG(\phi, \psi)$, when interpreted in an infinite tree, will mean `there exists a $\psi$-fair branch (i.e. an infinite path on which $\psi$ is true infinitely often) where $\phi$  always holds'.%
\footnote{Although similar in spirit, our `fair $\CTL$' is not the same as `$\CTL$ with fairness constraints' (called $\textup{FCTL}$  in~\cite{EmersonL86}), because in the latter fairness constraints are fixed once and for all as global external constraints, and do not recursively change inside a formula.}
This operator can be characterized as a greatest post-fixpoint of a $\CTL$-formula using the `until' connective $\EU$. The natural candidate axiomatization for fair $\CTL$ therefore consists of suitable fixpoint axioms and rules for these operations. 
In Section~\ref{s:completeness} we prove (Theorem~\ref{t:completeness}) that this candidate axiomatization is in fact complete with respect to the intended models. This result is obtained via a non-trivial tableaux procedure, adapting ideas already introduced to give a partial proof of completeness for the modal $\mu$-calculus in~\cite{Koz1983}, combined with some basic notions and techniques from modal logic and Stone duality.

Using this result, as our {\bf second contribution} we prove (Theorem~\ref{t:main1}) that the class of algebras corresponding to the logic fair $\CTL$ has a model companion. Moreover, this model companion can be axiomatized using the conversion of monadic second order logic into the modal $\mu$-calculus and back to bisimulation invariant monadic second order logic \cite{JanWal1996}. As in the case of linear temporal logic sketched above, a main ingredient is that the acceptance condition of the appropriate class of automata (in this case, $\mu$-automata) is expressible as an existential formula, using the new operators in fair $\CTL$.

For our {\bf third contribution}, we consider {\it binary} fair $\CTL$, i.e., the logic obtained from fair $\CTL$ by adding two deterministic modalities and an axiom saying that the `next' operator $\diam$ is the union of these two. We prove (Theorem~\ref{t:main2}) that the model companion for the class of binary fair $\CTL$-algebras `is' the monadic second order logic S2S; more precisely, it is the first-order theory of the powerset Boolean algebra of the full infinite binary tree.\\

\intpart{Paper outline}
The paper is organized as follows: in Section~\ref{s:fairCTL}, we introduce fair $\CTL$, its syntax, its semantics and some 
variants. Section~\ref{s:completeness} proves completeness theorems by means of suitable tableau constructions, relying on definable contextual connectives.
In Section~\ref{s:modelcompanions}, we 
show our results 
about
existence of model companions and their relationships with monadic second order logic.
Section~\ref{s:conclusions} concludes.
For space reasons, most proofs are omitted; details can be found in the 
appendix to this paper.

%

\section{CTL with fairness constraints}\label{s:fairCTL}
In this section, we introduce the logic `fair CTL', $\CTLf$ for short, which is a variant of the computation tree logic CTL with fairness constraints built in.
\heading{The logic $\CTLf$}
We introduce syntax (Def.~\ref{d:syntax}), semantics (Def.~\ref{d:semantics}), and an axiomatization (Def.~\ref{d:ctlf}) for the temporal logic $\CTLf$.
\begin{definition}(Syntax of $\CTLf$.) \label{d:syntax}
The \emph{basic operation symbols of $\CTLf$} are $0$-ary symbol $\bot$, unary symbols $\neg$ and $\diam$, and binary symbols $\vee$, $\EU$ and $\EG$. We define the following \emph{derived operations}:
\begin{itemize} 
\item $a \wedge b := \neg(\neg a \vee \neg b)$,
\item $\Box a := \neg \diam \neg a$,
\item $\AR(a,b) := \neg \EU(\neg a,\neg b)$, and
\item  $\AF(a,b):= \neg \EG(\neg a, \neg b)$.
\end{itemize}
Let $\p = \{p_1,\dots,p_n\}$ be a finite set of propositional variables. By a \emph{$\CTLf$-formula with variables in $\p$} we mean a term built up inductively by applying operation symbols of $\CTLf$ to propositional variables $p \in \p$. We denote by $\CTLf(\p)$ the set of $\CTLf$-formulas with variables in $\p$.
\end{definition}
$\CTLf$-formulas can be interpreted in \emph{transition systems}, as follows.
\begin{definition}(Semantics of $\CTLf$.) \label{d:semantics}
A \emph{transition system} is a pair $(S,R)$, where $S$ is a set and $R$ is a binary relation on $S$. An \emph{$R$-path} is a (finite or infinite) sequence of nodes $s_i \in S$ such that $s_i{R}s_{i+1}$ for all $i$. Whenever $R$ is clear from the context, we omit it and refer to the transition system as $S$, and to $R$-paths as paths. For $\p$ a set of variables, a \emph{$\p$-colouring} of a transition system $S$ is a function $\sigma : S \to \mathcal{P}(\p)$.

Let $(S,R,\sigma)$ be a $\p$-coloured transition system. The \emph{forcing relation}, $\forces$, between nodes $s \in S$ and formulas $\phi \in \CTLf(\p)$ is inductively defined as follows:
\begin{itemize}
\item $s \not\forces \bot$,
\item $s \forces p$ iff $p \in \sigma(s)$,
\item $s \forces \neg \psi$ iff $s \not\forces \psi$,
\item $s \forces \psi_1 \vee \psi_2$ iff $s \forces \psi_1$ or $s \forces \psi_2$,
\item $s \forces \diam \psi$ iff there exists $s' \in S$ such that $s{R}s'$ and $s' \forces \psi$,
\item $s \forces \EU(\psi_1,\psi_2)$ iff there exist $n \geq 0$ and an $R$-path $s = s_0,\dots,s_n$ such that $s_t \forces \psi_2$ for all $t < n$ and $s_n \forces \psi_1$.
\item $s \forces \EG(\psi_1,\psi_2)$ iff there exists an infinite $R$-path $s = s_0,s_1,\dots$ such that $s_t \forces \psi_1$ for all $t$ and there exist infinitely many $t$ with $s_t \forces \psi_2$.
\vspace{-5mm}
\end{itemize}
\end{definition}
\begin{remark}
For the derived operations, $\Box$, $\AR$ and $\AF$, we have,
\begin{itemize}
\item $s \forces \Box \psi$ iff for all $s' \in S$ such that $s{R}s'$, $s' \forces \psi$,
\item $s \forces \AR(\psi_1,\psi_2)$ iff for all $n \geq 0$ and all $R$-paths $s = s_0,\dots,s_n$, either $s_t \forces \psi_2$ for some $t < n$, or $s_n \forces \psi_1$.
\item $s \forces \AF(\psi_1,\psi_2)$ iff for all infinite $R$-paths $s = s_0,s_1,\dots$ such that there exist infinitely many $t$ with $s_t \not\forces \psi_2$, there exists $t$ such that $s_t \forces \psi_1$.\vspace{-5mm}
\end{itemize}
\end{remark}

\begin{convention}
We henceforth assume that \emph{all transition systems are serial}, i.e., for every $s \in S$, there exists $s' \in S$ such that $s{R}s'$; equivalently, $\diam\top$ is forced in all nodes.
\end{convention}
In order to axiomatize our logic, we now introduce the quasi-equational theory $\CTLf$.
\begin{definition}\label{d:ctlf}
The \emph{quasi-equational theory $\CTLf$} is axiomatized by the following finite set of quasi-equations\footnote{Here, and in what follows, we use the usual notation that `$a \leq b$' abbreviates `$a \vee b = b$'.}:
\begin{enumerate}
\item[(i)] Boolean algebra axioms for $\bot,\neg,\vee$,
\item[(ii)] (Axioms $\mathbf{K}$) $\diam \bot = \bot$,~~$\forall a, b : \diam(a \vee b) = \diam a \vee \diam b$,
\item[(iii)] (Axiom $\mathbf{D}$) $\diam \top = \top$,
\item[(iv)] (Fixpoint axioms) $\forall a, b, c : $
\begin{align}
&a \vee (b \wedge \diam \EU(a,b)) \leq \EU(a,b) %
\tag{$\EU_{\mathrm{fix}}$}, %
\label{EUfix}\\
&[a \vee (b \wedge \diam c) \leq c] \quad \to \quad [\EU(a,b) \leq c] %
\tag{$\EU_{\mathrm{min}}$}, %
\label{EUmin}\\
&\EG(a,b) \leq a \wedge \diam \EU(b \wedge \EG(a,b),a) %
\tag{$\EG_{\mathrm{fix}}$}, %
\label{EGfix} \\
&[c \leq a \wedge \diam \EU(b \wedge c,a)] \quad \to \quad [c \leq \EG(a,b)] %
\tag{$\EG_{\mathrm{max}}$}. %
\label{EGmax}
\end{align}
\end{enumerate}
\end{definition}
The models of the quasi-equational theory $\CTLf$ will be called $\CTLf$-algebras; we explicitly record the definition here.
\begin{definition}
\label{d:ctlfalgebra}
A \emph{$\CTLf$-algebra} is a tuple $$\b{A} = (A,\bot,\vee,\neg,\diam,\EU,\EG)$$ such that 
\begin{enumerate}
\item[(i)] the reduct $(A,\bot,\vee,\neg)$ is a Boolean algebra;
\item[(ii)] $\diam : A \to A$ is a unary operation that preserves finite joins, including the empty join, $\bot$;
\item[(iii)] $\diam \top = \top$;
\item[(iv)] $\EU$ and $\EG$ are binary operations on $A$ such that, for any $a, b \in A$,
\begin{itemize}
\item $\EU(a,b)$ is the least pre-fixpoint of the function $x \mapsto a \vee (b \wedge \diam x)$, and 
\item $\EG(a,b)$ is the greatest post-fixpoint of the function $y \mapsto a \wedge \diam \EU(b \wedge y, a)$.\vspace{-5mm}
\end{itemize}
\end{enumerate}
\end{definition}
This quasi-equational theory $\CTLf$ and its associated class of $\CTLf$-algebras can be used to define a modal logic, in the following standard way.
\begin{definition}
Let $\p = \{p_1,\dots,p_n\}$ be a finite set of propositional variables. A \emph{valuation} of $\p$ in a $\CTLf$-algebra $\b{A}$ is a function $V: \p \to A$.
For any $\CTLf$-formula $\phi(\p)$ and valuation $V$ in a $\CTLf$-algebra $\b{A}$, we write $\phi^\b{A}(V(\p))$ for the \emph{interpretation} of $\phi$ in the $\CTLf$-algebra $\b{A}$ under the valuation $V$.

An equation $\phi(\p) = \psi(\p)$ of $\CTLf$-formulas is called \emph{valid} if, and only if, it interprets to a true statement under any valuation of the propositional variables $\p$ in any $\CTLf$-algebra. Two $\CTLf$-formulas are \emph{equivalent} if the equation $\phi = \psi$ is valid. A $\CTLf$-formula $\phi$ is called a \emph{tautology} if $\phi = \top$ is a valid equation, and $\emph{consistent}$ if $\phi = \bot$ is not a valid equation;
a $\CTLf$-formula $\phi$ is said to \emph{entail} a formula $\psi$ (written  $\phi\vdash \psi$ or $\phi\leq \psi$) iff the formula $\neg \phi \vee \psi$ is a tautology. 
\end{definition}
Notice that, for the derived operations $\AR$ and $\AF$ (Def.~\ref{d:syntax}), we have 
\begin{itemize}
\item $\AR(a,b) = \max \{c \in A \ | \ c \leq a \wedge (b \vee \Box c)\}$,
\item $\AF(a,b) = \min \{c \in A \ | \ a \vee \Box\AR(b \vee c, a) \leq c\}$,
\end{itemize}
i.e., the following fixpoint rules hold for $\AR$ and $\AF$:
\begin{align}
&\AR(a,b) \leq a \wedge (b \vee \Box \AR(a,b)) \tag{$\AR_{\mathrm{fix}}$} \label{ARfix} \\
&[c \leq a \wedge (b \vee \Box c)] \quad \to \quad [c \leq \AR(a,b)] \tag{$\AR_{\mathrm{max}}$} \label{ARmax} \\
&a \vee \Box \AR(b \vee \AF(a,b),a) \leq \AF(a,b) \tag{$\AF_{\mathrm{fix}}$} \label{AFfix} \\
&[a \vee \Box \AR(b \vee c,a) \leq c] \quad \to \quad [\AF(a,b) \leq c]. \tag{$\AF_{\mathrm{min}}$} \label{AFmin}
\end{align}

\begin{remark}
A \emph{modal algebra} is a tuple $(A,\bot,\vee,\wedge,\neg,\diam)$ for which (i) and (ii) in Definition~\ref{d:ctlfalgebra} hold. The requirement in Definition~\ref{d:ctlfalgebra}(iii) that $\diam\top = \top$ says that $\b{A}$ is in fact an algebra for the modal logic $\mathbf{KD}$. The operations $\EU$ and $\EG$ of a $\CTLf$-algebra $\b A$ are uniquely determined by its modal algebra reduct. However, the operations $\EU$ and $\EG$ do not exist in every modal algebra.
\end{remark}

\heading{Semantics via $\CTLf$-algebras}
The following example and proposition connect the semantics of $\CTLf$ introduced in Def.~\ref{d:semantics} with the definition of $\CTLf$-algebras in Def.~\ref{d:ctlfalgebra}.
\begin{example}
The \emph{complex algebra} of a transition system $(S,R)$ is the tuple 
$$\b{P}(S) = (\mathcal{P}(S), \emptyset, \cup, S \setminus (-), \diam_R, \EU_R, \EG_R),$$
where
$(\mathcal{P}(S), \emptyset, \cup, S \setminus (-))$ is the Boolean power set algebra of the set $S$, \[ \diam_R(a) := R^{-1}[a] = \{s \in S \ | \ \text{ there exists } t \in a \text{ such that } s{R}t\},\]
and $\EU_R$ and $\EG_R$ are the unique binary operations making $\b{P}(S)$ into a $\CTLf$-algebra. (Indeed, such operations exist because $\mathcal{P}(S)$ is a complete lattice.)
\end{example}
Notice that $\p$-colourings $\sigma : S \to \mathcal{P}(\p)$ correspond bijectively to valuations $V : \p \to \mathcal{P}(S)$: given $\sigma$, we define $V_\sigma(p) := \{s \in S \ | \ p \in \sigma(s)\}$ for each $p$ in $\p$, and conversely, given $V$, we define $\sigma_V(v) := \{p \in \p \ | \ v \in V(p)\}$.
\begin{proposition}\label{p:compalg}
Let $(S,R,\sigma)$ be a $\p$-coloured transition system. For any $\CTLf(\p)$-formula $\phi$ and $s \in S$, we have
\[ s \forces \phi \iff s \in \phi^{\b{P}(S)}(V_\sigma(\p)).\]
\end{proposition}

\heading{Adding roots and binary determinism}
For later use, we define two expansions of the system $\CTLf$, $\CTLfp$ and $\CTLfpb$. For $\CTLfp$, we add one basic constant, $\I$, whose intended interpretation is to be true in exactly one `root' node in the transition system. For $\CTLfpb$, we add two additional basic operations, $\X_0$ and $\X_1$, whose intended interpretations are a deterministic `step left' and `step right' in the transition system.
\begin{definition}(Syntax of $\CTLfp$ and $\CTLfpb$.)\label{d:syntax2}
Let $\p$ be a set of proposition letters. We define the set $\CTLfp(\p)$ of \emph{rooted $\CTLf$-formulas} by adding one basic nullary operation $\I$ to $\CTLf$. 
We define the set $\CTLfpb(\p)$ of \emph{rooted binary $\CTLf$-formulas} by adding two basic unary operations, $\X_0$ and $\X_1$, to $\CTLfp(\p)$.
\end{definition}

\begin{definition}(Semantics of $\CTLfp$ and $\CTLfpb$.)\label{d:semantics2}
A node $s_0$ in a transition system $(S,R)$ is called a \emph{root} if for every $s \in S$ there is a path from $s_0$ to $s$, and there is no path ending in $s_0$ except for the trivial path consisting of only $s_0$. A transition system is called \emph{rooted} if it has a (necessarily unique) root.

If $(S,R)$ is a transition system with root $s_0$, we extend the definition of the forcing relation of $\CTLf(\p)$ (Def.~\ref{d:semantics}) to $\CTLfp(\p)$ by defining the additional base case
\begin{itemize}
\item $s \forces \I$ iff $s = s_0$.
\end{itemize}

A \emph{binary transition system} is a tuple $(S,R,f_0,f_1)$ such that $(S,R)$ is a transition system, and $f_0, f_1 : S \to S$ are unary functions such that $R = f_0 \cup f_1$. If $(S,R,f_0,f_1)$ is a rooted binary transition system, we extend the definition of the forcing relation of $\CTLfp(\p)$ to $\CTLfpb(\p)$ by defining, for $i = 0,1$,
\begin{itemize}
\item $s \forces \X_i \phi$ iff $f_i(s) \forces \phi$.
\end{itemize}
\vspace{-5mm}
\end{definition}
We now axiomatize the additional operations $\I$, $\X_0$ and $\X_1$, as follows.
\begin{definition}\label{d:ctlfp}
The \emph{universal theory} $\CTLfp$ is obtained by adding to the theory $\CTLf$ (Def.~\ref{d:ctlf}) the sentences
\begin{enumerate}
\item[(v)] (Axioms for $\I$)
\begin{itemize}
\item $\I \neq \bot$,
\item $\diam \EU(\I,\top) = \bot$,
\item $\forall a : [a \neq \bot] \to [\I \leq \EU(a,\top)]$.
\end{itemize}
\end{enumerate}
Models of $\CTLfp$ are called \emph{rooted $\CTLf$-algebras}; concretely, these are pairs $(\b{A},\I)$ where $\b{A}$ is a $\CTLf$-algebra and $\I \in A$ satisfies the axioms in (v). 

The \emph{universal theory} $\CTLfpb$ is obtained by adding to the theory $\CTLfp$ the sentences
\begin{enumerate}
\item[(vi)] (Axioms for $\X_0, \X_1$) 
\begin{itemize}
\item $\diam a = \X_0 a \vee \X_1 a$,
\end{itemize}
and, for $i = 0,1$:
\begin{itemize}
\item $\X_i \bot = \bot$, $\forall a, b : \X_i (a \vee b) = \X_i a \vee \X_i b$, 
\item $\X_i \neg a = \neg \X_i a$.
\end{itemize}
\end{enumerate}
Models of $\CTLfpb$ are called \emph{binary rooted $\CTLf$-algebras}; concretely, these are tuples $(\b{A},\I,\X_0,\X_1)$ where $(\b{A},I)$ is a rooted $\CTLf$-algebra and $\X_0, \X_1$ are unary operations on $A$ satisfying the axioms in (vi).
\end{definition}
The \emph{complex algebra} of a rooted transition system $(S,R)$ with root $s_0$ is obtained by expanding the complex algebra $\b{P}(S)$ of the transition system with the constant $\I := \{s_0\}$. 
The \emph{complex algebra} of a rooted binary transition system is obtained by further expanding this algebra with unary operations $\X_0$ and $\X_1$ defined, for $i = 0,1$ and $a \in \mathcal{P}(S)$, by
\[ \X_i a := f^{-1}(a) = \{s \in S \ | \ f(s) \in a\}.\]
Note that the analogue of Proposition~\ref{p:compalg} holds for $\CTLfp$ and $\CTLfpb$.

\begin{example}\label{ex:binarytree}
Let $S$ be the set of finite sequences of $0$'s and $1$'s, i.e., $S := 2^*$. For $i \in \{0,1\}$, let $f_i(w) := wi$, the sequence obtained by appending the symbol $i$ to the end, and let $R := f_0 \cup f_1$ be the `child' relation. Then $(S,R,f_0,f_1)$ is a rooted binary transition system, called the \emph{full binary tree}, with root the empty sequence $\epsilon$.
\end{example}

\section{Completeness}\label{s:completeness}

In this section we prove that our axiomatization of $\CTLf$ is \emph{complete} with respect to tree-shaped transition systems. Algebraically, this will mean that complex algebras of such transition systems generate the whole quasi-variety of 
$\CTLf$-algebras; a result that will be used several times to establish our main results in Section~\ref{s:modelcompanions}. The key theorem in this section, Thm.~\ref{t:completeness} below, shows that 
every consistent $\CTLf$-formula can be satisfied in a tree-shaped transition system. 

This result, and its variants for rooted and binary $\CTLf$-algebras, require a rather technical and laborious tableau construction. Readers who are only interested in the bigger picture may skip details in this section; the statements of Theorems~\ref{t:completeness}, \ref{t:rootedcompleteness} and \ref{t:binarycompleteness} are sufficient for continuing.

We first recall the definition and fix notation for trees.
\begin{definition}
A \emph{tree} is a rooted transition system $(S,R)$ such that for every $s$ in $S$, there is a \emph{unique} path from the root to $s$. A tree naturally comes with a \emph{partial order} $\preceq$, which is defined as the reflexive transitive closure of $R$, and has the property that $v \preceq v'$ iff $v$ lies on the unique path from the root to $v'$.
\end{definition}
As with transition systems, we will often suppress the notation of the transition relation $R$, and simply speak of a tree $S$. We are mostly concerned with infinite trees, and we will always specify it explicitly if a tree is finite. As with transition systems, if we only say `tree', then the tree is assumed to be serial, hence infinite.
We will prove the following theorem.
\begin{theorem}\label{t:completeness}
For every consistent $\CTLf$-formula $\phi_0(\p)$, there exists a $\p$-coloured tree with root $s_0$ such that $s_0 \forces \phi_0$.
\end{theorem}
In the rest of this section, we fix a consistent $\CTLf$-formula $\phi_0(\p)$. Since $\phi_0$ is consistent, we may also fix a $\CTLf$-algebra $\b{A}$ and an interpretation $V : \p \to \b{A}$ of proposition letters such that $\phi_0^{\b{A}}(V(\p)) \neq \bot$. We will use these data to construct a $\p$-coloured tree $S$, for which we will prove that $\phi_0$ holds in the root.

\begin{convention}\label{c:phi}
Since $\b{A}$, $\p$ and $V$ are fixed throughout the rest of this section, we will usually omit reference to them. In particular, if $\phi(\p)$ is a $\CTLf$-formula, then we will denote the interpretation of $\phi$ in $\b{A}$ under $V$ by $\phi$ as well, where we should actually write $\phi^{\b{A}}(V(\p))$ for that element.
\end{convention}

The proof of Theorem~\ref{t:completeness} will be structured as follows. In Subsection~\ref{ss:context}, we introduce a crucial syntactic tool that we call {\it contextual operations}. In Subsection~\ref{ss:preliminaries}, we then recall several other more standard preliminary notions that play a role in the proof: negation normal form, representation of modal algebras, Fischer-Ladner closure, and types. The heart of the construction of the $\p$-coloured tree $S$ is in Subsection~\ref{ss:model}, where we inductively construct the tree as a union of partial tableaux. 

In Subsection~\ref{ss:compvariants}, we will state the analogous completeness theorems for the variants $\CTLfp$ and $\CTLfpb$.

\subsection{Contextual operations and rules}\label{ss:context}
The following syntactic definition is crucial to the completeness proof. The meaning of these operations will be clarified in the rest of this subsection.
\begin{definition}\label{d:EUcAFc}
We introduce the following ternary operations $\EU_c$ (`contextual $\EU$') and $\AF_c$ (`contextual $\AF$') as abbreviations of term operations in $\CTLf$:
\begin{itemize}
\item $\EU_c(p,q,r) := p \vee (q \wedge \diam \EU(p \wedge r,q \wedge r))$,
\item $\AF_c(p,q,r) := \AF(p,q) \wedge (p \vee \Box \AR(q\vee r,p))$.
\end{itemize}
\vspace{-5mm}
\end{definition}
In Proposition~\ref{p:EUcAFcfixpoint}, we will show that $\EU_c$ and $\AF_c$ can be characterized as least fixpoints of operators very similar to those for $\EU$ and $\AF$ (cf. Def~\ref{d:ctlf} and further). The only difference is that, in the contextual versions of $\EU$ and $\AF$, the proposition in the third coordinate is added conjunctively to the fixpoint variable. The third coordinate may therefore be thought of as a `context', hence the name. 
This idea (although not under this name) originates with the partial completeness proof for the modal $\mu$-calculus in \cite{Koz1983}. The additional piece of information that we prove here is that the contextual versions of $\EU$ and $\AF$ are themselves still expressible in $\CTLf$, which is of course only a fragment of the full modal $\mu$-calculus.
\begin{proposition}\label{p:EUcAFcfixpoint}
For any elements $p,q,r$ of a $\CTLf$-algebra $\b A$, we have:
\begin{enumerate}
\item $\EU_c(p,q,r)$ is the least pre-fixpoint of the monotone function $x \mapsto p \vee (q \wedge \diam (r \wedge x))$, and
\item $\AF_c(p,q,r)$ is the least pre-fixpoint of the monotone function $x \mapsto p \vee \Box \AR(q \vee (r \wedge x),p)$.
\end{enumerate}
\end{proposition}
\begin{remarkconv}\label{r:equiexp}
Note that, for any $p,q$, we have 
$$
\EU_c(p,q,\top) = \EU(p,q)~\text{and}~ \AF_c(p,q,\top) = \AF(p,q)~.
$$ 
Thus, in the syntax of $\CTLf$, we can replace the operator $\EU$ by $\EU_c$ and the operator $\AF$ by $\AF_c$, and obtain an equi-expressive formalism. For this reason, and
in this section only, \emph{we will drop the subscript `$c$' and simply use the notations $\EU$ and $\AF$ for both the ternary and the binary versions} of these operators. Any `binary' occurrence $\EU(\phi,\psi)$ or $\AF(\phi,\psi)$ should be read as $\EU(\phi,\psi,\top)$ or $\AF(\phi,\psi,\top)$, respectively. Formally, this is only a syntactic convenience, but it turns out to be very useful in  the completeness proof.
\end{remarkconv}

The reason for introducing the contextual operations is the following lemma that we refer to as a `context rule'. This is the version of \cite[Prop. 5.7(vi)]{Koz1983} that we need here.
\begin{proposition}\label{p:contextrules}
For any elements $p,q,r,\gamma$ of a $\CTLf$-algebra $\b{A}$, we have
\begin{enumerate}
\item if $\gamma \wedge \EU(p,q,r) \neq \bot$, then $\gamma \wedge \EU(p,q,r\wedge \neg \gamma) \neq \bot$,
\item if $\gamma \wedge \AF(p,q,r) \neq \bot$, then $\gamma \wedge \AF(p,q, r \wedge \neg \gamma) \neq \bot$.
\end{enumerate}
\end{proposition}
%
%

\subsection{Other preliminary notions}\label{ss:preliminaries}
We recall and fix notation for negation normal form, representation of modal algebras via ultrafilters, types, and Fischer-Ladner closure.

\heading{Negation normal form}
It will be convenient to put $\CTLf$-formulas in \emph{negation normal form}.
\begin{definition}\label{d:nnf}
Let $\p$ be a finite set of propositional variables. The set of \emph{$\CTLf$-formulas in negation normal form} is defined via the following grammar:
\begin{align*}
&\phi \isbnf
   &\bot \divbnf \top \divbnf p \divbnf \neg p 
    \divbnf \diam \phi \divbnf \Box \phi \divbnf \phi \lor \phi \divbnf \phi \land \phi~~~~~~~~~~~~~~~~~
 \\&&  
   \divbnf \EU(\phi,\psi,\chi)      \divbnf \AR(\phi,\psi)
   \divbnf \EG(\phi,\psi) \divbnf \AF(\phi,\psi,\chi)
\end{align*}
Note that negation is only allowed to be applied to propositional variables. We do not need ternary connectives for $\AR$ and $\EG$.
\end{definition}
\begin{lemma}\label{l:nnf}
Any $\CTLf$-formula is equivalent to a $\CTLf$-formula in negation normal form.
\end{lemma}
Throughout the rest of this section, we assume all $\CTLf$-formulas are in negation normal form.

\heading{Representation of modal algebras}
We will make use of the following representation of the modal algebra underlying a $\CTLf$-algebra.
\begin{definition}\label{d:dualframe}
Let $\b{A}$ be a modal algebra. The \emph{dual frame} of $\b{A}$ is the pair $\b{A}_* = (A_*,R_*)$, where
\begin{itemize}
\item $A_*$ is the set of ultrafilters of the Boolean algebra $A$;
\item $R_*$ is the binary relation on $X$ defined by $x{R_*}y$ if, and only if, for every $a \in A$, if $a \in y$ then $\diam a \in x$.
\end{itemize}
\vspace{-5mm}
\end{definition}
\begin{theorem}\label{t:jtrep}\cite{JonTar1951}
Any modal algebra embeds in the complex algebra of its dual frame.
\end{theorem}
By contrast, not every $\CTLf$-algebras embeds into a complex $\CTLf$-algebra.
An important part of Theorem~\ref{t:jtrep} is worth recording separately.
\begin{lemma}\label{l:diamstep}
Let $\b{A}$ be a modal algebra with dual frame $\b{A}_*$. If $a \in A$, $x \in \b{A}_*$, and $\diam a \in x$, then there exists $y \in \b{A}_*$ such that $x {R_*} y$ and $a \in y$.
\end{lemma}
\heading{Types and characteristic formulas}
The following equivalence relations on the points of $\b{A}_*$, and characterizing formulas for them, will also be useful. In the following definition, recall that a point $x\in\b{A}_*$ is an ultrafilter of $\b{A}$  and so, under Convention~\ref{c:phi}, it makes sense to say that $\phi$ \emph{belongs to} $x$.
\begin{definition}\label{d:charform}
Let $\rho$ be a finite set of formulas. For any $x, x' \in \b{A}_*$, define
\[ x \sim_\rho x' \iff x \cap \rho = x' \cap \rho.\]
We call the equivalence class of a point $x$ under $\sim_\rho$ the \emph{$\rho$-type of $x$}.

For any $x \in \b{A}_*$, define the {\it characteristic formula}
\[\kappa(x,\rho) := \bigwedge \{\gamma \ | \ \gamma \in \rho \cap x\} \wedge \bigwedge \{\neg \gamma \ | \ \gamma \in \rho \setminus x\}.\vspace{-5mm}\]

\end{definition}
\begin{lemma}\label{l:charform}
For any set of formulas $\rho$ and points $x, x' \in \b{A}_*$, we have
\[ x \sim_\rho x' \iff \kappa(x,\rho) \in x'.\]
\end{lemma}

We  combine the above with Proposition~\ref{p:contextrules} to obtain the following useful fact, which will allow us, in the next subsection, to make `jumps' in the ultrafilter frame of $\b{A}$.
\begin{lemma}\label{l:jump}
Let $\rho$ be a finite set of formulas, let $\hs \in \{\EU,\AF\}$, and let $\phi$, $\psi$, and $\chi$ be formulas. For any $x \in \b{A}_*$ such that $\hs(\phi,\psi,\chi) \in x$, there exists $x' \in \b{A}_*$ such that $x \sim_\rho x'$ and $\hs(\phi,\psi,\chi \wedge \neg\kappa(x,\rho)) \in x'$.
\end{lemma}

\heading{Fischer-Ladner closure}
A last standard concept that we need in our construction is the Fischer-Ladner closure of a finite set of formulas.
\begin{definition}\label{d:closure}
A set of $\CTLf$-formulas $\Gamma$ is called (Fischer-Ladner) \emph{closed} if the following hold:
\begin{itemize}
\item $\EU(\top,\top,\top) \in \Gamma$,
\item if $\phi \in \Gamma$, then $\phi' \in \Gamma$ for any subformula $\phi'$ of $\phi$,
\item if $\EG(\phi,\psi) \in \Gamma$, then $\diam \EU(\psi \wedge \EG(\phi,\psi),\phi) \in \Gamma$.
\item if $\AR(\phi,\psi) \in \Gamma$, then $\Box \AR (\phi,\psi) \in \Gamma$.
\item if $\EU(\phi,\psi,\chi) \in \Gamma$, then $\diam(\chi \wedge \EU(\phi,\psi,\chi)) \in \Gamma$,
\item if $\AF(\phi,\psi,\chi) \in \Gamma$, then
$\Box \AR(\psi \vee \chi,\phi) \in \Gamma$.
\end{itemize}
The \emph{closure} of a set of $\CTLf$-formulas is the smallest closed set containing it.
\end{definition}
\begin{lemma}\label{l:finiteclosure}
The closure of a finite set of $\CTLf$-formulas is finite.
\end{lemma}

\subsection{Model construction}\label{ss:model}
Now that we have all the preliminaries in place, we will construct a tree for the consistent formula $\phi_0(\p)$ that we fixed above, based on the $\CTLf$-algebra $\b{A}$ and valuation $V: \p \to \b{A}$ (cf. Convention~\ref{c:phi} above). In what follows, $\Gamma_0$ denotes the Fischer-Ladner closure of $\{\phi_0\}$, which is finite by Lemma~\ref{l:finiteclosure}.

A standard model construction in modal logic would be to consider the quotient of the ultrafilter frame $\b{A}_*$ by the equivalence relation $\sim_{\Gamma_0}$. Our model construction is necessarily more intricate than that, because of the operators $\EU$ and $\AF$, which are defined as least fixpoints. Let us call an \emph{eventuality formula} a $\CTLf$-formula of the form $\hs(\phi,\psi,\chi)$, where $\hs \in \{\AF,\EU\}$. The \emph{set of eventuality formulas} in propositional variables $\p$ will be denoted by $\mathrm{Ev}(\p)$.

We will construct a tree $S$ as a union of finite trees. For  the construction of these finite trees, we use a notion of \emph{partial tableau for $\Gamma_0$ in $\b{A}$} (see Definition~\ref{d:partialtableau} below). Before giving the formal definition, we will explain the idea behind it.

A partial tableau for $\Gamma_0$ in $\b{A}$ will consist of a finite tree $T$ and two labellings, $\alpha$ and $\beta$. The labelling $\alpha$ will assign to each node of the finite tree $T$ an ultrafilter of $\b{A}$, which can be thought of as the set of formulas that we would like to force in that node. The labelling $\beta$ assigns to each node a data structure that records the `current status' of eventuality formulas in $\Gamma_0$. This data structure is a finite list of tuples of the form $(\theta,\sigma,\rho,\chip)$. Here, if the $k^\mathrm{th}$ element in the list $\beta(v)$ is $(\theta,\sigma,\rho,\chip)$, then $\theta = \hs(\phi,\psi,\chi)$ is an eventuality formula in $\Gamma_0$ which lies in $\alpha(v')$ for some $v' \preceq v$ (i.e. for some tree ancestor $v'$ of $v$); $\sigma$ is a `status' which can be either $\a$ (active), $\f$ (frozen) or $\e$ (extinguished); $\rho$ is a finite set of formulas that we call the `relevance set' and is used in the construction; and $\chip$ is a `context formula', which will be a strengthening of $\chi$. 
We now give the formal definition.

\begin{definition}\label{d:partialtableau}
Let $\Gamma_0$ be a finite closed set of $\CTLf$-formulas with variables in $\p$. Define
\[\Sigma := (\Gamma_0 \cap \mathrm{Ev}(\p)) \times \{\a,\f,\e\} \times \mathcal{P}_\mathrm{fin}(\CTLf(\p)) \times \CTLf(\p).\]

A \emph{partial tableau} for $\Gamma_0$ in $\b{A}$ is a tuple $(T,\alpha,\beta)$, where
\begin{itemize}
\item $T$ is a finite tree,
\item $\alpha$ is a function from $T$ to $\b{A}_*$, the set of ultrafilters of $\b{A}$,
\item $\beta$ is a function from $T$ to $\Sigma^*$, the set of finite words over $\Sigma$.
\end{itemize}

For each $v \in T$, we write $\ell(v)$ for the length of $\beta(v)$. For each $1 \leq k \leq \ell(v)$, we write $\beta(v)_k$ for the $k^\mathrm{th}$ letter of the word $\beta(v)$, and denote this letter by $(\theta(v)_k,\sigma(v)_k,\rho(v)_k,\chip(v)_k)$, where $\theta(v)_k = \hs(v)_k(\phi(v)_k,\psi(v)_k,\chi(v)_k)$ for some $\hs(v)_k \in \{\AF,\EU\}$ and formulas $\phi(v)_k$, $\psi(v)_k$ and $\chi(v)_k$.
\end{definition}

In accordance with the intuitive explanation of a partial tableau, we will also impose some {\it well-formedness} conditions on the partial tableau, namely (cf. Definition~\ref{d:wellformed} below): (a) any element in the list $\beta(v)$ persists in the list $\beta(v')$ for tree successors $v'$ of $v$; (b) if the first coordinate $\phi$ of an eventuality formula lies in $\alpha(v)$, then it is extinguished; (c) $\Gamma_0$ is always contained in the relevance set; (d) $\EU$-formulas can never be frozen; (e) $\chip$ is a strengthening of $\chi$; (f) eventuality formulas that occur at some earlier point in the list always lie in the relevance set; and (g) non-extinguished eventuality formulas must lie in $\alpha(v)$.
\begin{definition}\label{d:wellformed}
We say the partial tableau $(T,\alpha,\beta)$ for $\Gamma_0$ in $\b A$ is \emph{well-formed} if, for all $v \in T$ and $1 \leq k \leq \ell(v)$,
\begin{enumerate}
\item[(a)] for all $v' \in T$ such that $v \preceq v'$, we have $\ell(v) \leq \ell(v')$, and $\theta(v)_k = \theta(v')_k$;
\item[(b)] if $\phi(v)_k \in \alpha(v)$ then $\sigma(v)_k = \e$;
\item[(c)] $\Gamma_0 \subseteq \rho(v)_k$; 
\item[(d)] if $\hs(v)_k = \EU$, then $\sigma(v)_k \neq \f$;
\item[(e)] $\chip(v)_k \vdash \chi(v)_k$;
\item[(f)] if $k' < k$ then $\hs(v)_{k'}(\phi(v)_{k'},\psi(v)_{k'},\chip(v)_{k'}) \in \rho(v)_k$;
\item[(g)] if $\sigma(v)_k \neq \e$ then $\hs(v)_k(\phi(v)_k,\psi(v)_k,\chip(v)_k) \in \alpha(v)$.
\vspace{-6mm}
\end{enumerate}
\end{definition}

We will now describe how to unravel a well-formed partial tableau. Again, before giving the lengthy formal definition (Def.~\ref{d:onestepunravel}) of the one-step unravelling of a partial tableau, we give an intuitive explanation. 
Recall that $\b{A}_* = (A_*,R_*)$ denotes the ultrafilter frame of $\b A$ (Def.~\ref{d:dualframe}).
In a simple tableau construction, to unravel a node $v$, one would add successors for all $\diam$-formulas in $\Gamma_0 \cap \alpha(v)$ and label them by appropriate $R_*$-successors of $\alpha(v)$.
In order to treat eventuality formulas, we need to modify this construction in the following way.
Instead of using the successors of $\alpha(v)$ as labels of children of $v$, we make a `jump' in the ultrafilter frame $\b{A}_*$ from the point $\alpha(v)$ to a point $x_v$, guided by the first active eventuality formula, $\hs_m(\phi_m,\psi_m,\chi_m)$, in the list $\beta(v)$.
We will then label the children of $v$ not by $R_*$-successors of $\alpha(v)$, but by $R_*$-successors of $x_v$.
The precise choice of $x_v$ is guided by the relevance set $\rho_m$, and will ensure (i) that $\alpha(v)$ and $x_v$ have the same $\rho_m$-type, and (ii) that the negation of $\kappa(\alpha(v),\rho_m)$ can be added conjunctively to $\chip_m$, while keeping the partial tableau well-formed.
Such an $x_v$ will exist because of Lemma~\ref{l:jump}. The advantage of this construction is that $x_v$ will contain a stronger statement than $\hs_m(\phi_m,\psi_m,\chi_m)$, which will prevent that unwanted infinite loops occur in the construction (cf. Lemma~\ref{l:noloops} below).

\begin{definition}\label{d:onestepunravel}
We define the \emph{one-step unravelling} of a well-formed 
partial tableau $(T,\alpha,\beta)$. For each leaf $v$ of $T$, add a finite set of children of $v$, $C_v := \{w_\lambda \ | \ \diam \lambda \in \Gamma_0 \cap \alpha(v)\}$.\footnote{Note that $C_v \neq \emptyset$, since $\diam(\top \wedge \EU(\top,\top,\top)) \in \Gamma_0$ because $\Gamma_0$ is closed, and $\diam(\top \wedge \EU(\top,\top,\top)) = \top$ in $\b{A}$.} We will now specify a value for $\alpha$ and $\beta$ on each of these children. 

Fix a leaf $v$.\footnote{In the rest of this definition, we mostly suppress notation for $v$, and in particular write $\theta_k$, $\sigma_k$, $\rho_k$, etc. instead of $\theta(v)_k$, $\sigma(v)_k$, $\rho(v)_k$, etc.} To define the values of $\alpha$ and $\beta$ on $C_v$, we first choose an auxiliary ultrafilter $x_v \in \b{A}_*$.
If $\sigma_k \neq \a$ for all $1 \leq k \leq \ell(v)$, define $x_v := \alpha(v)$. Otherwise, put 
\[m := \min \{1 \leq k \leq \ell(v) \ | \ \sigma_k = \a\}.\]
We call $m$ the \emph{active index} at $v$.%
\footnote{If $m$ does not exist, proceed as in the case $\hs_m = AF$ for the definition of $\alpha$, and in the definition of $\beta$ act as if $m = \infty$.} 
By Def.~\ref{d:wellformed}(g), we have $\hs(\phi_m,\psi_m,\chip_m) \in \alpha(v)$. Therefore, by Lemma~\ref{l:jump}, pick $x_v \in \b{A}_*$ such that $\hs_m(\phi_m,\psi_m,\chip_m \wedge \neg \kappa(\alpha(v),\rho_m)) \in x_v$ and $x_v \sim_{\rho_m} \alpha(v)$. Write $\gamma_v := \kappa(\alpha(v),\rho_m)$.

Let $w_\lambda \in C_v$. We use $x_v$ to define $\alpha(w_\lambda)$ and $\beta(w_\lambda)$. For the definition of $\alpha(w_\lambda)$, the cases $\hs_m = \AF$ and $\hs_m = \EU$ diverge slightly.

\begin{itemize}
\item {\bf Case $\hs_m = \AF$.} Since $\diam \lambda \in \Gamma_0 \cap \alpha(v)$, we have $\diam \lambda \in x_v$, because $\Gamma_0 \subseteq \rho_m$ and $\alpha(v) \sim_{\rho_m} x_v$. Therefore, by Lemma~\ref{l:diamstep}, pick $\alpha(w_\lambda)$ such that $x_v {R_*} \alpha(w_\lambda)$ and $\lambda \in \alpha(w_\lambda)$.
\item {\bf Case $\hs_m = \EU$.} We do the same as in the previous case if $\lambda \neq \chi_m \wedge \EU(\phi_m,\psi_m,\chi_m)$. If $\lambda = \chi_m \wedge \EU(\phi_m,\psi_m,\chi_m)$, we do the following. By Def.~\ref{d:wellformed}(b) and $\sigma_m = \a$, we have $\phi_m \not\in \alpha(v)$. Since $\alpha(v) \sim_{\rho_m} x_v$ and $\phi_m \in \Gamma_0 \subseteq \rho_m$, we have $\phi_m \not\in x_v$, so $\neg \phi_m \in x_v$. Also, $\EU(\phi_m,\psi_m,\chip_m \wedge \neg \gamma_v) \in x_v$ by the choice of $x_v$. 
Applying the general fact (Proposition~\ref{p:EUcAFcfixpoint}) that $\EU(p,q,r) \wedge \neg p \leq \diam (r \wedge \EU(p,q,r))$, we obtain 
$\diam (\chip_m \wedge \neg \gamma_v \wedge \EU(\phi_m,\psi_m,\chip_m \wedge \neg \gamma_v)) \in x_v$. 
By Lemma~\ref{l:diamstep}, pick $\alpha(w_\lambda)$ such that $x_v{R_*}\alpha(w_\lambda)$ and $\chip_m \wedge \neg \gamma_v \wedge \EU(\phi_m,\psi_m,\chip_m \wedge \neg \gamma_v) \in \alpha(w_\lambda)$. Note that in particular $\chi_m \wedge \EU(\phi_m,\psi_m,\chi_m) \in \alpha(w_\lambda)$, since $\chip_m \wedge \neg \gamma_v \leq \chi_m$ and $\EU$ is monotone.
\end{itemize}
The word $\beta(w_\lambda)$ is defined as an update of the word $\beta(v)$, obtained by consecutively applying the following steps:
\begin{enumerate}
\item Let $\mathsf{New}(w_\lambda) := \{ \theta \in \alpha(w_\lambda) \cap \Gamma_0 \cap \mathrm{Ev}(\p) \ | \ \forall 1 \leq k \leq \ell(v) : \text{ if } \theta_k(v) = \theta, \text{ then } \sigma_k(v) = \e\}$\footnotemark
\footnotetext{Note that $\mathsf{New}(w_\lambda)$ is non-empty, because it always contains the formula $\EU(\top,\top,\top)$.}. 
%
For each $\theta = \hs(\phi,\psi,\chi) \in \mathsf{New}(w_\lambda)$, add one letter, $(\theta,\a,\rho',\chi)$, to the end of the word, where $\rho' := \bigcup_{k=1}^{\ell(v)} \rho_k$.
\item For each position $k$, put
\[ \chip(w_\lambda)_k = \begin{cases} \chip(v)_k &\mbox{if } k < m, \\
									  \chip(v)_m \wedge \neg \gamma_v & \mbox{if } k = m, \\
									  \chi(v)_k &\mbox{if } k > m.\end{cases} \]
\item For each position $k > m$, add the formula $\hs_m(\phi_m,\psi_m,\chip(v)_m \wedge \neg \gamma_v)$ to the set $\rho_k$.
\item For each position $k$ such that $\phi_k \in \alpha(w_\lambda)$, change $\sigma_k$ into $\e$.
\item For each position $k$, if $\theta_k = \EU(\phi_k,\psi_k,\chi_k)$ and $\lambda \neq \chi_k \wedge \EU(\phi_k,\psi_k,\chi_k)$, change $\sigma_k$ into $\e$.
\item For each position $k$, if $\hs_k = \AF$, $\psi_k \in \alpha(w_\lambda)$, and $\sigma_k = \a$, change $\sigma_k$ into $\f$.
\item For each position $k < m$, if $\hs_k = \AF$, $\sigma_k = \f$, $\phi_k \not\in \alpha(w_\lambda)$ and $\psi_k \not\in \alpha(w_\lambda)$, change $\sigma_k$ into $\a$.
\vspace{-6mm}
\end{enumerate}
\end{definition}

\begin{lemma}\label{l:wellformed}
The one-step unravelling of a well-formed partial tableau $T$ is well-formed. 
\end{lemma}


\begin{definition}\label{d:modelconstruction}
We define a tree $(S,R)$ with a $\p$-colouring $\sigma : S \to \mathcal{P}(p)$.
Since $\phi_0 \neq \bot$, pick an ultrafilter $x_0 \in \b{A}_*$ such that $\phi_0 \in x_0$. Define $(T_0,\alpha_0,\beta_0)$ to be the partial tableau whose underlying tree consists of a single node, $s_0$, and $\alpha_0(s_0) := x_0$. Choose a word $\beta_0(s_0)$ which orders (in an arbitrary manner) the set $\{(\hs(\phi,\psi,\chi),\a,\Gamma_0,\chi) \ | \ \hs(\phi,\psi,\chi) \in x_0 \cap \Gamma_0 \cap \mathrm{Ev}(\p), \phi \not\in x_0\}$. Note that $(T_0,\alpha_0,\beta_0)$ is a well-formed partial tableau.
Inductively define, for each $n \geq 0$, $(T_{n+1},\alpha_{n+1},\beta_{n+1})$ to be the one-step unravelling of the partial tableau $(T_n,\alpha_n,\beta_n)$. 

Let $(S,R)$ be the infinite, finitely branching tree $\bigcup_{n=0}^\infty T_n$. We have well-defined functions $\alpha := \bigcup_{n=0}^\infty \alpha_n$ from $S$ to $\b{A}_*$ and $\beta := \bigcup_{n=0}^\infty \beta_n$ from $S$ to $\Sigma^*$. For each $v \in S$, define $\sigma(v) := \alpha(v) \cap \p$.
\end{definition}
From the truth lemma below, it follows 
that $s_0 \forces \phi_0$, which  concludes the proof of Theorem~\ref{t:completeness}:
\begin{lemma}[Truth Lemma]\label{l:truthlemma}
For all $\theta \in \Gamma_0$ and $v \in S$, if $\theta \in \alpha(v)$, then $v \forces \theta$.
\end{lemma}
The proof of the truth lemma is, as usual, by induction on $\theta$. The induction makes use of the following crucial fact, showing that eventualities are 
always extinguished (i.e. fulfilled) or, in the case of $\AF$, ultimately frozen along any branch: 
\begin{lemma}\label{l:noloops}
For all $v \in S$, $1 \leq k \leq \ell(v)$, and infinite $R$-paths $(v_t)_{t=0}^\infty$ with $v_0 = v$, there exists $t \geq 0$ such that 
\begin{itemize}
\item if $\hs(v)_k = \EU$, then $\sigma(v_t)_k = \e$.
\item if $\hs(v)_k = \AF$, then either $\sigma(v_t)_k = \e$, or for all $t' \geq t$, $\sigma(v_{t'})_k = \f$.
\end{itemize}
\end{lemma}

\subsection{Completeness for rooted and binary fair CTL}\label{ss:compvariants}
We conclude the section by stating the completeness results for our variants of $\CTLf$. The proofs are mild modifications of the model construction given in Subsection~\ref{ss:model} above.

\begin{theorem}\label{t:rootedcompleteness}
For every consistent $\CTLfp$-formula $\phi_0(\p)$, there exists a $\p$-coloured tree such that for some node $s$, $s \forces \phi_0$.
\end{theorem}
%
%
\begin{theorem}\label{t:binarycompleteness}
For every consistent $\CTLfpb$-formula $\phi_0(\p)$, there exists a $\p$-colouring $\sigma$ of the full binary tree such that for some node $s$, $s \forces \phi_0$.
\end{theorem}

\section{Model companions}\label{s:modelcompanions}
The aim of this section is to exhibit model companions for the universal theories $\CTLfp$ and $\CTLfpb$ (Def.~\ref{d:ctlfp}) of rooted $\CTLf$-algebras and rooted binary $\CTLf$-algebras. 

We first recall the formal definition of model companion from model theory. For more conceptual background on the notion of model companion, we refer to the introduction of this paper and, e.g., \cite[Section 3.5]{CK} and \cite{Whe1976}.

\begin{definition}\label{d:modelcompanion}
A first-order theory $T^*$ is \emph{model-complete} if every formula is equivalent over $T^*$ to an existential formula\footnote{ 
In fact, it is sufficient that every \emph{universal} formula is equivalent over $T^*$ to an existential one, see~\cite[Thm.~3.5.1]{CK}.
}. 

A first-order theory $T^*$ is a \emph{co-theory} of a first-order theory $T$ if every model of $T$ embeds into a model of $T^*$, and vice versa.

Let $T$ be a universal theory. An extension $T^* \supseteq T$ is a \emph{model companion} of $T$ iff $T^*$ is a model-complete co-theory of $T$.
\end{definition}
It can be shown~\cite[Section 3.5]{CK} that a model companion $T^*$ - whenever it exists - is unique and axiomatizes the class of models of $T$ which are \emph{existentially closed} for $T$. Recall that a $T$-model $M$ is existentially closed for $T$ iff, whenever an existential formula $\phi$ with parameters from $M$ holds in a $T$-model $M' \supseteq M$, then $\phi$ holds in $M$ itself.

A remark on notation is in place. In this section, we will mainly be concerned with the \emph{first-order theory} of rooted $\CTLf$-algebras. We will denote the (functional) first-order language of rooted $\CTLf$-algebras by $\mathcal{L}$. Thus, $\mathcal{L}$ has function symbols $\bot,\vee,\neg,\I,\diam,\EU$, and $\EG$, one relation symbol, $=$, and the usual first-order connectives. In contrast with the previous section, the word `formula' (or `$\mathcal{L}$-formula') will here refer to a first-order formula in this language $\mathcal{L}$, and we will use lower case Greek letters $\phi$, $\psi$, etc. for these. To avoid any possible confusion, in this section we refer to $\CTLf$-formulas as \emph{$\mathcal{L}$-terms}, and we use lower case Roman letters $t$, $u$, etc. for these.\footnote{Note that propositional connectives such as $\bot$, $\neg$, $\vee$, etc. can have two distinct meanings when they occur in an $\mathcal{L}$-formula: they are used to build $\mathcal{L}$-terms, as in, e.g., $\neg \I \vee p$, but they are also symbols of the first-order meta-language, as in, e.g., $\neg(p = q)$. Thus, the two occurrences of `$\neg$' in the $\mathcal{L}$-formula $\neg(\top = (\neg \I \vee p))$ have different meanings. In practice, we will parenthesize carefully to avoid confusion.}

A straight-forward but important observation about the theory $\CTLfp$ is that quantifier-free formulas reduce to equations, or inequations.
\begin{lemma}\label{l:justequations}
For any quantifier-free $\mathcal{L}$-formula $\phi(\overline{p})$, there exists an $\mathcal{L}$-term $t_\phi(\overline{p})$ such that $\mathrm{CTL}^f_\mathrm{I} \proves \phi \leftrightarrow (t_\phi = \top)$; similarly, there exists an $\mathcal{L}$-term $t_\phi'(\overline{p})$ such that $\mathrm{CTL}^f_\mathrm{I} \proves \phi \leftrightarrow (t_\phi' \neq \bot)$.
\end{lemma}

\subsection{$\CTLfp$ has a model companion: proof outline}\label{ss:proofoutline}
We shall construct, in Subsection~\ref{ss:Tstar}, a first-order theory that we call $(\CTLfp)^*$, and prove (Thm.~\ref{t:main1}) that $(\CTLfp)^*$ is the model companion of $\CTLfp$. In this subsection we give a general outline of the proof.\\

\intpart{Construction of the theory $(\CTLfp)^*$}
For the theory $(\CTLfp)^*$ to be model-complete, we will need that any universal formula is equivalent over $(\CTLfp)^*$ to an existential one. By Lemma~\ref{l:justequations}, any universal formula is equivalent over $\CTLfp$ to a universal formula of the particular form $\forall \x \, t(\p,\x) = \top$, where $t$ is an $\mathcal{L}$-term. 
We will construct, for each such special universal formula $\phi(\p)$, an existential formula $\psi(\p)$ with the following two properties:
\begin{enumerate}
\item[(I)] $\CTLfp \proves \forall \p (\psi(\p) \rightarrow \phi(\p))$, and
\item[(II)] any rooted $\CTLf$-algebra with $\p$-parameters extends to a model where $\phi(\p) \to \psi(\p)$ is true.
\end{enumerate}
The formula $\psi(\p)$ with these two properties will allow us to construct the model companion of $\CTLfp$. \\

\intpart{Construction of $\psi$} We now outline the construction of the existential formula $\psi(\p)$ mentioned in the construction of the theory $(\CTLfp)^*$ above. 
For this, we use the back-and-forth translation between formulas of the modal $\mu$-calculus and automata by Janin and Walukiewicz \cite{JanWal1995,JanWal1996}. The process will go in three steps:
\begin{enumerate}
\item[Step 1.] From a first-order $\mathcal{L}$-formula $\phi(\p)$ to a monadic second order formula $\Phi(\p)$ (Prop.~\ref{p:termtoMSO});
\item[Step 2.] From a  monadic second order formula $\Phi(\p)$ to a non-deterministic modal automaton $\mathcal{A}$, which describes the behaviour of $\Phi(\p)$ on $\omega$-expansions of trees (Prop.~\ref{p:MSOtoaut}); 
\item[Step 3.] Back from the automaton $\mathcal{A}$ to an $\mathcal{L}$-term $\acc_{\mathcal{A}}(\p,\q)$ describing the automaton (Prop.~\ref{p:auttoterm}).
\end{enumerate} 
The $\mathcal{L}$-term $\acc_{\mathcal{A}}(\p,\q)$, once the variables $\q$ corresponding to the states of the automaton are existentially quantified, is transformed into the existential formula 
$\psi(\p):= \exists \q\,(\acc_{\mathcal{A}}(\p,\q) =\top)$. This will be the existential formula $\psi$ mentioned in the construction of $(\CTLfp)^*$ above.

\subsection{Obtaining an existential formula using automata}\label{ss:exform}
In this subsection we make the construction of $\psi$, outlined in the previous subsection, precise. For this purpose, we first recall the definitions of $\omega$-expansions and fix the notation that we use for $\MSO$. After this, we give the technical results underlying Step 1 -- 3 in the construction of $\psi$.\\

\intpart{$\omega$-expansions of trees}
The following definition actually works for transition systems in general, cf. \cite[Def.~1]{JanWal1996}, but we only need it for trees.

\begin{definition}
Let $(S,R)$ be a tree with root $s_0$. The \emph{$\omega$-expansion}, $(S_\omega,R_\omega)$, of $(S,R)$ is the tree which is defined as follows:
\begin{align*}
& S_\omega := \{(k_1,s_1) \dots (k_n,s_n) \in (\omega \times S)^* \ | \  s_i {R} s_{i+1}~(0 \leq i < n)\},\\
& R_\omega[(k_1,s_1) \cdots (k_n,s_n)] := \{(k_1,s_1)\cdots(k_n,s_n)(k_{n+1},s_{n+1}) \ : 
          \\ &~~~~~~~~~~~~~~~~~~~~~~~~~~~~~~~~~~~~~~~~~~~~~~~~ :
              \ k_{n+1} \in \omega, s_n {R} s_{n+1}\}.
\end{align*}
We denote the empty sequence by $\epsilon$; note that $\epsilon$ is the root of $(S_\omega,R_\omega)$. Also note that the definition of $S_\omega$ requires in particular that, if $(k_1,s_1) \dots (k_n,s_n) \in S_\omega$, then $s_0{R}s_1{R}\dots{R}s_n$ is a finite path in $S$ starting at the root.

For any $\p$-colouring $\sigma$ of a tree $(S,R)$ with root $s_0$, define the $\p$-colouring $\sigma_\omega$ of $(S_\omega,R_\omega)$ by $\sigma_\omega(\epsilon) := \sigma(s_0)$, and 
\[\sigma_\omega((k_1,s_1)\dots(k_n,s_n)) := \sigma(s_n).\vspace{-5mm}\]
\end{definition}

It is straight-forward to prove that 
any $\p$-coloured tree is bisimilar to its $\omega$-expansion via a back-and-forth morphism. Stating this in algebraic terms, we have in particular:

\begin{proposition}\label{p:subalgebra}
For any $\p$-coloured tree $(S,\sigma)$, the algebra $\b{P}(S)$ is isomorphic to a subalgebra of $\b{P}(S_\omega)$, via an isomorphism which in particular sends $V_\sigma(p)$ to $V_{\sigma_\omega}(p)$ for each $p$ in $\p$.
\end{proposition}


\intpart{$\MSO$ on trees}
We use the following (reduced) \emph{syntax of monadic second order logic}  $\MSO$. The \emph{atomic formulas} of $\MSO$
are of the form $p \subseteq q$ and $R(p,q)$ where $p$, $q$ are variables; arbitrary formulas are obtained from atomic formulas using the connectives $\vee$, $\neg$ and $\exists p$. This syntax suffices to express all of $\MSO$, cf., e.g., \cite[p. 7]{Tho1996} or \cite[Ch. 12]{GTW2002}. In particular, 
we use the abbreviation $p=q$ for $(p\subseteq q)\wedge (q\subseteq p)$ and we use the convention that lower case letters $v, v', 
\dots$ stand for `individual variables', whose interpretation is forced to be a singleton. For a first-order variable $v$ and a second-order variable $p$, we write `$v \in p$' to mean `$v \subseteq p$'.

As for the semantics, we will only consider interpretations of  $\MSO$ over  trees
$(S,R)$: given an $\MSO$ formula $\Phi(\p)$ and a $\p$-colouring $\sigma : S \to \mathcal{P}(\p)$ with associated valuation $V_{\sigma} : \p \to \mathcal{P}(S)$, 
the relation $S, \sigma \models_{\MSO} \Phi(\p)$ is defined in the usual way, i.e., for atomic formulas we have
\begin{align*}
 S,\sigma \models_{\MSO} p\subseteq q~~ \iff &~ V_{\sigma}(p)\subseteq V_{\sigma}(q)\\
 S,\sigma \models_{\MSO} R(p, q) \iff &~ R\cap (V_{\sigma}(p)\times V_{\sigma}(q)) \neq \emptyset,
\end{align*}
and this definition is extended to arbitrary $\MSO$-formulas.\\

\intpart{Step 1: From $\FO$ to $\MSO$} The essence of the following proposition is the so-called `standard translation' from modal fixpoint logic to monadic second-order logic.
\begin{proposition}\label{p:termtoMSO}
For any first-order $\mathcal{L}$-formula $\phi(\p)$, there exists a monadic second order formula $\Phi(\p)$ such that, for any $\p$-coloured tree $(S,\sigma)$, 
\[ \b{P}(S),V_\sigma \models_{\FO} \phi(\p) \iff S, \sigma \models_{\MSO} \Phi(\p).\]
\end{proposition}
\begin{proof} (Sketch) We just show how $\Phi$ is built up.
We first inductively define, for any $\mathcal{L}$-term $t(\p)$, an MSO-formula $\dot{t}(\p,v)$, where $v$ is a fresh first-order variable. 
The base case and the cases for the function symbols other than $\EU$ and $\EG$ are treated as follows:
\begin{multicols}{2}
\begin{itemize}
\item $\dot{p_i} := v \in p_i$,
\item $\dot{(t_1 \vee t_2)} := (\dot{t_1}(v)) \vee (\dot{t_2}(v))$,
\item $\dot{(\neg t_1)} := \neg (\dot{t_1}(v))$,
\item $\dot{\bot} := \neg (v = v)$,
\item $\dot{\diam t} := \exists v' (R(v,v') \wedge \dot{t}(\p,v'))$,
\item $\dot{\I} := \forall v' (\neg R(v',v)).$
\end{itemize}
\end{multicols}
Before defining $\dot{\EU(t_1,t_2)}$ and $\dot{\EG(t_1,t_2)}$, first define the auxiliary formula:
\begin{align*}
& \mathsf{Pre}_{t_1,t_2}(p,q) := 
\\ &~~~~~~~ \forall v' \left( ([\dot{t_1}(v') \wedge p(v')] \vee [\dot{t_2}(v') \wedge R(v',q)]) \to (v' \in q) \right)~.
\end{align*}
Note that $\mathsf{Pre}_{t_1,t_2}(p,q)$ is true in a transition system $S$ exactly if $(t_1 \wedge p) \vee (t_2 \wedge \diam q) \leq q$ holds in the algebra $\mathbb{P}(S)$.

We now define:
\[ \dot{\EU(t_1,t_2)} := \forall q \left( \mathsf{Pre}_{t_1,t_2}(v'=v',q) \to q(v)\right),\]
in words: $\EU(t_1,t_2)$ is forced in $v$ iff $v$ lies in all sets $q$ for which $t_1 \vee (t_2 \wedge \diam q) \leq q$.

We also define:
\begin{equation*}
\dot{\EG(t_1,t_2)} := \exists p \left(
\begin{aligned}
&  (v \in p) \wedge \forall v' [(v' \in p) \to \\
&(\dot{t_1}(v') \wedge \exists v''[R(v',v'')  \\
&\wedge \forall q \left(\mathsf{Pre}_{t_2,t_1}(p,q) \to (v'' \in q)\right)] )]
\end{aligned}
\right)
\end{equation*}
In words: $\EG(t_1,t_2)$ is forced in $v$ iff $v$ lies in some set $p$ such that $p \leq t_1 \wedge \diam \EU(t_2 \wedge p, t_1)$ holds.


Now, for any $\mathcal{L}$-formula $\phi(\p)$, define $\Phi(\p)$ by replacing any atomic formula $t_1 = t_2$ by $\forall v (\dot{t_1}(\p,v) \leftrightarrow \dot{t_2}(\p,v))$. 
\end{proof}

\intpart{Step 2: From $\MSO$ to automata} We recall the relevant definitions and results from \cite{JanWal1995,JanWal1996}. The details will be relevant in Step 3 as well.
\begin{definition}\label{d:automaton}
Fix a finite set $\p$ of propositional variables. A \emph{non-deterministic modal automaton over $\p$} is a tuple $\mathcal{A} = (Q,q_0,\delta,\Omega)$, where $Q$ is a finite set of states, $q_0 \in Q$ is an initial state, $\delta : Q \times \mathcal{P}(\p) \to \mathcal{P}\mathcal{P}(Q)$ is a transition function, and $\Omega : Q \to \omega$ is a parity function.

Let $(S_\omega,\sigma_\omega)$ be the $\omega$-expansion of a $\p$-coloured tree $(S,\sigma)$. A \emph{successful run} of the automaton  $\mathcal{A}$ on $(S_\omega,\sigma_\omega)$ (also known as \emph{$\mathcal{A}$-labelling} in \cite{AgoHol2000}) is a function $r : S_\omega \to Q$ such that:
\begin{enumerate}
\item (Initial) $r(\epsilon) = q_0$,
\item (Transition) for all $v \in S_\omega$, the set $\{r(v') \ | \ v{R}v'\}$ is in $\delta(r(v),\sigma_\omega(v))$,
\item (Success) for any infinite path $(v_t)_{t=0}^\infty$ in $S_\omega$ with $v_0 = \epsilon$, the parity
\[ \min \{\Omega(q) \ | \ r(v_t) = q \text{ for infinitely many } t \in \omega\}\]
is even.
\end{enumerate}

We say that $\mathcal{A}$ \emph{accepts} $(S_\omega,\sigma_\omega)$ if there exists a successful run of $\mathcal{A}$ on $(S_\omega,\sigma_\omega)$.
\end{definition}
Note that we only gave the definition of acceptance of an \emph{$\omega$-expanded tree}. We do not need the more involved acceptance condition for general trees.

\begin{proposition}\label{p:MSOtoaut}
For any monadic second order formula $\Phi(\p)$, there exists a non-de\-ter\-mi\-nistic modal automaton $\mathcal{A}_\Phi$ over $\p$ such that, for any $\p$-coloured tree $(S,\sigma)$, 
\[ (S_\omega,\sigma_\omega) \models \Phi(\p) \iff \mathcal{A}_{\Phi} \text{ accepts } (S_\omega,\sigma_\omega). \]
\end{proposition}
\begin{proof}
By \cite[Lem.~12]{JanWal1996}, there is a formula $\Phi^\vee(\p)$ of the modal $\mu$-calculus such that for every $\p$-coloured tree $(S,\sigma)$, 
\[(S_\omega,\sigma_\omega) \models \Phi \iff (S,\sigma) \models \Phi^\vee.\]
Since any $\p$-coloured tree is bisimilar to its $\omega$-expansion,
we also have
\[ (S, \sigma) \models \Phi^\vee \iff (S_\omega,\sigma_\omega) \models \Phi^\vee.\]
By the results in \cite{JanWal1995} (also see, e.g., \cite[Sec.~2]{AgoHol2000}), there is a non-de\-ter\-mi\-nistic modal automaton $\mathcal{A}_\Phi$ such that 
\[(S_\omega,\sigma_\omega) \models \Phi^\vee \iff \mathcal{A}_{\Phi} \text{ accepts } (S_\omega,\sigma_\omega).\vspace{-5mm}\]
\end{proof}

\intpart{Step 3: From automaton to term}
Here, we use the fact that the language of $\CTLfp$ is expressive enough to express the acceptance condition of automata on $\omega$-expanded trees. In particular, we  need the binary $\AF$ connective of $\CTLf$ for the term $\acc_3$ in the proof.
\begin{proposition}\label{p:auttoterm}
For any non-deterministic modal automaton $\mathcal{A}$ over $\p$ with set of states $\q$, there exists an $\mathcal{L}$-term $\acc_{\mathcal{A}}(\p,\q)$ such that for any $\p$-coloured tree $(S,\sigma)$, we have
\[ \mathcal{A} \text{ accepts } (S_\omega,\sigma_\omega) \iff \b{P}(S_\omega),V_{\sigma_\omega} \models \exists \q \; \acc_{\mathcal{A}}(\p,\q) = \top. \]
\end{proposition}
\begin{proof} (Sketch)
We encode acceptance conditions into an $\mathcal{L}$-term $\acc_{\mathcal{A}}(\p,\q)$.
We define the following auxiliary terms for $D \in \mathcal{P}(\q)$ and $\alpha \in \mathcal{P}(\p)$:
\[ \nabla D := \bigwedge_{q \in D} \diam q \wedge \Box \left(\bigvee_{q \in D} q\right) \quad \text{ and } \quad \odot\!\alpha := \bigwedge_{p \in \alpha} p \wedge \bigwedge_{p\not\in\alpha}\neg p.\]
Now  the required $\mathcal{L}$-term $\acc_{\mathcal{A}}(\p,\q)$ is taken to be $\acc_1 \wedge \acc_2 \wedge \acc_3$, where  
\begin{align*}
\acc_1(\p,\q) &:= \neg \I \vee q_0,\\
\acc_2(\p,\q) &:= \bigvee_{q \in \q}\left( 
  \begin{aligned}
    & q \wedge \bigwedge_{q' \in \q \setminus\{q\}} \neg q' \wedge 
    \\ &
    \bigvee\left\{\nabla D \wedge \odot \alpha \ | \  \alpha \in \mathcal{P}(\p), D \in \delta(q,\alpha)\right\}
  \end{aligned}
\right),\\
\acc_3(\p,\q) &:= \bigwedge \left\{\AF\left(\bigvee_{\Omega(q')<n}q',\bigwedge_{\Omega(q)=n}\neg q\right) \right\},
\end{align*}
where the last conjunction is taken over the set of the odd numbers
 $n$ that belongs to the range of $\Omega$. 
%
%
\end{proof}

\subsection{The model companion of $\CTLfp$}\label{ss:Tstar}

Let
$J$ be the set of triples 
$(t,\p,\x)$ such that  $\p$  and $\x$ are  disjoint finite sets of variables and $t$ is an $\mathcal{L}$-term in variables  $\p \cup \x$.
For each tuple $j = (t,\p,\x) \in J$, define the first-order $\mathcal{L}$-formula
\[ \phi_j(\p) := \forall \x \; t(\p,\x) = \top,\]
and let $\Phi_j(\p)$ be the monadic second-order formula given by Proposition~\ref{p:termtoMSO}.
%
Define the first-order $\mathcal{L}$-formula
\begin{equation}\label{eq:psij}
 \psi_j(\p) := \exists \q \; \mathsf{acc}_{\mathcal{A}_{\Phi_j}}(\p,\q) = \top,
\end{equation}
where $\mathcal{A}_{\Phi_j}$ is the non-deterministic automaton corresponding to the MSO formula $\Phi_j$, by Proposition~\ref{p:MSOtoaut}. 

Finally, define the first-order $\mathcal{L}$-theory
\[ (\CTLfp)^* := \CTLfp \cup \{\forall \p (\phi_j \to \psi_j) \ | \ j \in J\}.\]


We now come to our main theorem.
\begin{theorem}\label{t:main1}
$(\CTLfp)^*$ is the model companion of $\CTLfp$.
\end{theorem}

\begin{proof} (Sketch)
 In order to show that $(\CTLfp)^*$ is model-complete one shows, by using the completeness theorem (Thm.~\ref{t:rootedcompleteness}),
  that for each $j = (t,\p,\x) \in J$, in all rooted $\CTLf$-algebras we have
\begin{equation*}
\forall \p \,( \psi_j \rightarrow \phi_j ).
\end{equation*}
This corresponds to property (I) in the proof outline given in Subsection~\ref{ss:proofoutline}.
Indeed, given this fact,  it follows from the definition of $(\CTLfp)^*$  that every universal formula is equivalent in $(\CTLfp)^*$ to an existential one, so that $(\CTLfp)^*$ is model-complete.
 
 That $(\CTLfp)^*$ is a co-theory of $\CTLfp$ follows from Lemma~\ref{l:extendalgebra} below, which corresponds to property (II) in the proof outline given in Subsection~\ref{ss:proofoutline}.
\end{proof}

\begin{lemma}\label{l:extendalgebra}
Let $j = (t,\p,\x) \in J$, with $\p = p_1,\dots,p_n$. For any rooted $\CTLf$-algebra $\b{A}$ and $\oa \in A^n$, there is a rooted $\CTLf$-algebra $\b{A}'$ which contains $\b{A}$ as a subalgebra such that $\b A' \models \phi_j(\oa) \to \psi_j(\oa)$.
\end{lemma}

\begin{rema}\label{r:nomodels}
\rm Although we have explicitly defined a model companion for $\CTLfp$, the models of the model companion $(\CTLfp)^*$ itself remain rather mysterious.
For instance, the only atom in a model of $(\CTLfp)^*$ is $\I$, as can be seen by 
taking as $\phi(p)$ the formula  $\forall x ((x\leq p) \to [(x=\bot) \vee (x=p)])$. \hfill $\lhd$
%
\end{rema}

\subsection{The binary case}\label{ss:modelcomp-binary}

We shall prove that $\CTLfpb$ has a model-companion too and, in addition, we shall be able to characterize this model-companion as the first-order theory of the complex algebra of the full binary tree. \\

\intpart{$\mathrm{S2S}$: $\MSO$ on the binary tree} Recall that $\mathrm{S2S}$ is the monadic second order logic of the full binary tree $2^*$ (we refer to~\cite[p. 7]{Tho1996} or \cite[Ch. 12]{GTW2002} for basic 
results used below). From a syntactic point of view, in $\mathrm{S2S}$ we have, in addition to the atomic formulas of $\MSO$, also the atomic formulas 
$f_0(p,q)$ and $f_1(p,q)$. Semantically, we interpret $\mathrm{S2S}$-formulas $\phi(\p)$ over $\p$-colourings $\sigma:\sigma \colon \p \to \mathcal{P}(2^*)$
of the full binary tree; for $i=0,1$, the new atomic formula $f_i(p,q)$ is interpreted so that we have 
\[ 
2^*, \sigma\models f_i(p,q) \iff f_i\cap (V_{\sigma}(p)\times V_{\sigma}(q))\neq \emptyset,
\]
 where on the right hand side, we view the unary function $f_i$ of Example~\ref{ex:binarytree} as a set of pairs.

Write $\mathcal{L}_{0,1}$ for the first-order language of rooted binary $\CTLf$-algebras. The following is proved in the same way as Prop.~\ref{p:termtoMSO}.

\begin{proposition}\label{prop:b-fo-to-s2s}
For any $\mathcal{L}_{0,1}$-formula $\phi(\p)$, there exists an $\mathrm{S2S}$-formula $\Phi(\p)$ such that, for any $\p$-colouring
$\sigma \colon \p \to \mathcal{P}(2^*)$, we have
\[\mathbb{P}(2^*), V_{\sigma} \models_{\FO} \phi(\p) \iff 2^*,\sigma \models_{\mathrm{S2S}} \Phi(\p).\]
\end{proposition}

\intpart{Encoding automata} 
A \emph{parity tree automaton} $\mathcal{A}$ on a finite alphabet $\Sigma$ is a tuple $(Q, q_I, \Delta, \Omega)$ where $Q$ is a finite set of states, $q_I \in Q$, $\Delta \subseteq Q \times \Sigma \times Q \times Q$, and $\Omega: Q\longrightarrow \omega$. We shall consider only automata whose alphabet $\Sigma$ is of the kind 
$\mathcal{P}(\p)$ for a finite set of propositional letters $\p$.
If $\sigma \colon 2^* \to \Sigma:=\mathcal{P}(\p)$ is a $\p$-colouring of the  tree $2^*$, then a \emph{run} of $\mathcal{A}$ on $\sigma$ is a function $\rho : 2^* \to Q$ such that $\rho(\epsilon) = q_I$ and, for any $w \in 2^*$, $(\rho(w), \sigma(w), \rho(w0), \rho(w1)) \in \Delta$. If $\pi \in 2^\omega$ is an infinite branch, we denote by $\mathrm{Inf}_\rho(\pi)$ the set of states in $Q$ which occur infinitely often on $\pi$ in $\rho$, i.e., $\mathrm{Inf}_\rho(\pi) := \{q \in Q \ | \ \pi \cap \rho^{-1}(q) \text{ is infinite}\}$. A run $\rho$ of $\mathcal{A}$ on $\sigma$ is \emph{successful} if for every infinite branch $\pi$ we have that $\min(\{\Omega(q)\mid q\in Inf_\rho(\pi)\})$ is even. We say $\mathcal{A}$ \emph{accepts} a $\p$-colouring $\sigma$ iff there exists a successful run of $\mathcal{A}$ on $\sigma$. The following result is well-known. For a proof, cf., e.g., \cite[Thm.~8.7 \& Lem.~12.21]{GTW2002}.

\begin{theorem}\label{thm:b-s2s-to-automaton}
Let $\Phi(\p)$ be a formula of $\mathrm{S2S}$. There exists a parity tree automaton $\mathcal{A}$ over the alphabet $\Sigma:=\mathcal{P}(\p)$ such that, for any $\sigma \colon 2^* \to \mathcal{P}(\p)$, 
\[ 2^*, \sigma \models_{\mathrm{S2S}} \Phi \iff \mathcal{A} \text{ accepts } \sigma.\]
\end{theorem}

Analogous to Prop.~\ref{p:auttoterm}, we also have:
\begin{proposition}\label{p:b-auttoterm}
For any parity tree automaton $\mathcal{A}=(Q, q_I, \Delta, \Omega)$ over $\Sigma:=\mathcal{P}(\p)$ with set of states $\q$, there exists an $\mathcal{L}_{0,1}$-term $\acc_{\mathcal{A}}(\p,\q)$ such that for any $\p$-colouring $\sigma: 2^*\longrightarrow \mathcal{P}(\p)$, we have
\[ \mathcal{A} \text{ accepts } (2^*,\sigma) \iff \b{P}(2^*),V_{\sigma} \models \exists \q \; \acc_{\mathcal{A}}(\p,\q) = \top. \]
\end{proposition}

Putting together what we have, we conclude that 

\begin{theorem}\label{t:main2}
The first-order theory $(\CTLfpb)^*$ of the binary tree algebra $\mathbb{P}(2^*)$ is the model companion of $\CTLfpb$.
\end{theorem}

\begin{proof}
 Let $\phi(\p)$ be a $\mathcal{L}_{0,1}$-formula. Using Proposition~\ref{prop:b-fo-to-s2s}, Theorem~\ref{thm:b-s2s-to-automaton} and Proposition~\ref{p:b-auttoterm},
 it is clear that 
 \[ (\CTLfpb)^* \proves \phi(\p) \leftrightarrow \exists \q \; \acc_{\mathcal{A}}(\p,\q) = \top~. \]
 Thus, every formula is equivalent modulo $(\CTLfpb)^*$ to an existential formula, so $(\CTLfpb)^*$ is model-complete.
 
 To show that $\CTLfpb$ and $(\CTLfpb)^*$ are co-theories, since $\CTLfpb \subseteq (\CTLfpb)^*$, it is sufficient to show that every rooted binary $\CTLf$-algebra embeds into a model of $(\CTLfpb)^*$, i.e., into an algebra which is elementarily equivalent to $\mathbb{P}(2^*)$. By compactness and Robinson Diagram
 Lemma (cf.~\cite[Prop.~2.1.8]{CK}), it is sufficient to prove the consistency of the union of $(\CTLfpb)^*$ with a finite conjunction $\phi$ of ground literals with parameters in the support of $\b{A}$ such that $\b{A} \models \phi$. For this, in view of Proposition~\ref{l:justequations}, it is sufficient to show that $(\CTLfpb)^* \cup\{t'_{\phi}(\overline{a})
 \neq \bot\}$
 has a model for some term $t'_\phi(\p)$ such that $\b{A}\models t'_{\phi}(\overline{a})
 \neq \bot$. The latter means that $t'_\phi(\p)$ is a consistent rooted binary $\CTLf$-formula, so we can simply invoke  
 the completeness Theorem~\ref{t:binarycompleteness} to get what we need. 
 \end{proof}

\section{Conclusion}\label{s:conclusions}
%
%
There is an important difference between our results for the tree logic $\CTLfp$ and the \emph{binary} tree logic $\CTLfpb$. In the binary case, we know that the model companion $(\CTLfpb)^*$ 
of $\CTLfpb$  is the first-order theory of the powerset algebra of the full binary tree. In contrast, by Remark~\ref{r:nomodels}, no powerset algebra can be a model of $(\CTLfp)^*$. 
From this, we can conclude that, if one wants to find a framework for $\MSO$ on infinite trees where `all equations are solvable', complex algebras of transition systems are insufficient and 
 algebraic models  become indispensable.

We leave to further research the interesting questions, posed by
the reviewers, whether $\CTL$ itself has a model companion, and
which are the minimal algebraizable fragments of the modal mu-calculus
having a model companion.



\acks

We wish to thank Yde Venema for his invaluable help and his decisive suggestions, and the organizers of the 2015 Dagstuhl seminar on Duality in Computer Science for facilitating this exchange. Most of this research was performed while the second-named author was at the University of Milan, supported by the project FIRB ``Futuro in Ricerca'' (RBFR10DGUA--002). We are very grateful to Vincenzo Marra for making this possible. Finally, we thank the reviewers for their many constructive comments and questions.
%


\bibliographystyle{abbrvnat}






\newpage
\onecolumn

\appendix

{\huge \bf Appendix}

\makeatletter
\renewcommand{\section}{%
  \@startsection
    {section}%
    {1}%
    {0pt}%
    {-\@sectionaboveskip}%
    {\@sectionbelowskip}%
    {\Large \bfseries \raggedright}}
\renewcommand{\subsection}{%
  \@startsection%
    {subsection}%
    {2}%
    {0pt}%
    {-\@subsectionaboveskip}%
    {\@subsectionbelowskip}%
    {\large \bfseries \raggedright}}
\makeatother

\section{Proofs for Section~\ref{s:fairCTL}}
{\large The following lemma is crucial for proving Prop.~\ref{p:compalg}.

\begin{lemma}\label{l:compalg-indstep}
Let $(S,R)$ be a transition system with complex algebra 
\[\b{P}(S) = (\mathcal{P}(S), \emptyset, \cup, S \setminus (-), \diam, \EU, \EG).\]
For any $a_1, a_2 \in \mathcal{P}(S)$ and $s \in S$, we have
\begin{enumerate}
\item $s \in \EU(a_1,a_2)$ iff there exist $n \geq 0$ and an $R$-path $s = s_0,\dots,s_n$ such that $s_t \in a_2$ for all $t < n$ and $s_n \in a_1$.
\item $s \in \EG(a_1,a_2)$ iff there exists an infinite $R$-path $s = s_0,s_1,\dots$ such that $s_t \in a_1$ for all $t$ and there exist infinitely many $t$ with $s_t \in a_2$.
\end{enumerate}
\end{lemma}
\begin{proof}
1. By the definition of $\EU$ as a least pre-fixpoint, it suffices to prove that the set
\begin{align*}
x_0 &:= \{s \in S \ | \ \text{there exist } n \geq 0 \text{ and an $R$-path } s = s_0,\dots,s_n \text{ such that } s_t \in a_2 \text{ for all } t < n \text{ and } s_n \in a_1\}
\end{align*}
is the least element $x$ of $\b{P}(S)$ for which $a_1 \vee (a_2 \wedge \diam x) \leq x$ holds.

\begin{itemize}
\item $a_1 \vee (a_2 \wedge \diam x_0) \leq x_0$: If $s \in a_1$, then $s \in x_0$, as witnessed by the trivial path `$s$'. If $s \in a_2 \wedge \diam x_0$, pick an $R$-successor $s_1$ of $s$ such that $s_1 \in x_0$. Pick an $R$-path $s_1,\dots,s_n$ witnessing that $s_1 \in x_0$. Since $s \in a_2$, the $R$-path $s, s_1, \dots, s_n$ witnesses that $s \in x_0$.
\item $x_0$ is the least such: Suppose that $a_1 \vee (a_2 \wedge \diam x) \leq x$ for some $x \in \b{P}(S)$. We need to show that $x_0 \leq x$. Suppose that $s_0 \in x_0$ and choose an $R$-path $s_0,\dots,s_n$ witnessing this. We show by induction on $n$ that $s \in x$. For $n = 0$, then $s_0 \in a_1$, so $s_0 \in x$. For $n > 0$, the shorter path $s_1, \dots, s_n$ gives, by induction, that $s_1 \in x$. Now $s_0 \in a_2$ and $s_0 \in \diam x$, so $s_0 \in x$.
\end{itemize}

2. By the definition of $\EG$ as a greatest post-fixpoint, it suffices to prove that the set
\begin{align*} 
x_0 &:= \{ s \in S \ | \ \text{there exists an infinite $R$-path } s = s_0, s_1, \dots \text{ such that } s_t \in a_1 \text{ for all } t \text{ and } s_t \in a_2 \text { for infinitely many } t\}
\end{align*}
is the greatest element $x$ of $\b{P}(S)$ for which $x \leq a_1 \wedge \dia \EU(a_2 \wedge x,a_1)$ holds. 
\begin{itemize}
\item $x_0 \leq a_1 \wedge \dia \EU(a_2 \wedge x_0,a_1)$: let $s_0 \in x_0$, and pick an infinite $R$-path $s_0,s_1,\dots$ witnessing this. Clearly, $s_0 \in a_1$. Moreover, the $R$-successor $s_1$ of $s_0$ lies in $\EU(a_2 \wedge x_0,a_1)$: pick some $t \geq 1$ such that $s_t \in a_2$. The infinite $R$-path $s_t, s_{t+1},\dots$ witnesses that $s_t \in x_0$, so $s_t \in a_2 \wedge x_0$. We also have $s_{t'} \in a_1$ for all $1 \leq t' < t$, concluding the proof that $s_1 \in \EU(a_2 \wedge x_0, a_1)$ by item (1).
\item $x_0$ is the greatest such: let $x \in \b{P}(S)$ be such that $x \leq a_1 \wedge \dia \EU(a_2 \wedge x, a_1)$. To show $x \leq x_0$, let $s_0 \in x$ be arbitrary; we show that $s_0 \in x_0$. Since $s_0 \in \dia \EU(a_2 \wedge x,a_1)$, by definition of $\dia$ and item (1), pick a successor $s_1$ of $s$ and a finite $R$-path $s_1,\dots,s_n$ such that $s_i \in a_1$ for all $1 \leq i \leq n$ and $s_n \in a_2 \wedge x$. Repeating this argument, we obtain an infinite $R$-path $(s_i)_{i=0}^\infty$ witnessing that $s \in x_0$.\vspace{-5mm}
\end{itemize}
\end{proof}

\begin{repproposition}{p:compalg}
Let $(S,R,\sigma)$ be a $\p$-coloured transition system. For any $\CTLf(\p)$-formula $\phi$ and $s \in S$, we have
\[ s \forces \phi \iff s \in \phi^{\b{P}(S)}(V_\sigma(\p)).\]
\end{repproposition}

\begin{proof}
By induction on the complexity of $\phi$. All cases except $\EU$ and $\EG$ are immediate from the definitions. 

Suppose that $\phi = \EU(\psi_1,\psi_2)$ or $\phi = \EG(\psi_1,\psi_2)$. Write $a_k := \psi_k^{\b{P}(S)}(V_\sigma(\p))$ for $k = 1,2$. By the induction hypothesis, we have, for $k = 1,2$ and for all $s \in S$,
\[ s \forces \psi_k \iff s \in a_k.\]
By Lemma~\ref{l:compalg-indstep}, we obtain the desired equivalences for $\EU(\psi_1,\psi_2)$ and $\EG(\psi_1,\psi_2)$.
\end{proof}

\newpage
\section{Proofs for Section~\ref{s:completeness}}
\subsection{Proofs for Subsection~\ref{ss:context}}
The following facts are clear semantically, and not so hard to derive syntactically.
\begin{lemma}\label{l:basicsyntfacts}
For any elements $a,a',b,b',c$ in a $\CTLf$ algebra $\b{A}$,
\begin{enumerate}
\item $\EU(a \vee a',b) = \EU(a,b) \vee \EU(a',b).$
\item $\AR(a \wedge a', b) = \AR(a,b) \wedge \AR(a',b).$
\item $a \leq a'$ and $b \leq b'$ implies $\EG(a,b) \leq \EG(a',b')$.
\item $a \leq a'$ and $b \leq b'$ implies $\AF(a,b) \leq \AF(a',b')$.
\item $\AR(a,b) \wedge \neg b \leq \Box \AR(a,b).$
\item if $a \wedge c \leq \Box c$ and $b \wedge c \leq \Box c$, then $\EU(a,b)\wedge c \leq \EU(a \wedge \Box c,b \wedge \Box c)$.
\end{enumerate}
\end{lemma}
\begin{proof}
(1) The operator $\EU$ is clearly monotone, being the fixpoint of a monotone operation. It now suffices to prove $\EU(a \vee a',b) \leq \EU(a,b) \vee \EU(a',b)$. Write $c := \EU(a,b) \vee \EU(a',b)$. Distributing disjunctions over $\diam$ and conjunction with $b$, notice that
\[ (a \vee a') \vee (b \wedge \diam c) = [a \vee (b \wedge \diam \EU(a,b))] \vee [a' \vee (b \wedge \diam \EU(a',b))] \leq c,\]
where we use two applications of (\ref{EUfix}) for the last inequality. By (\ref{EUmin}), we conclude that $\EU(a \vee a', b) \leq c$, as required.

(2) follows from (1) since $\AR$ is the De Morgan dual of $\EU$.

(3) and (4) are clear since $\EG$ and $\AF$ are fixpoints of monotone operations.

(5) By (\ref{ARfix}), $\AR(a,b) \leq b \vee \Box \AR(a,b)$, from which the statement follows.

(6) Suppose that $a \wedge c \leq \Box c$ and $b \wedge c \leq \Box c$. Write $d := \neg c \vee \EU(a \wedge \Box c, b \wedge \Box c)$. We need to show that $\EU(a,b) \leq d$. For this, it suffices to prove, by (\ref{EUmin}), that $a \vee (b \wedge \diam d) \leq d$. Since $a \wedge c \leq a \wedge \Box c$ by assumption, we have $a \wedge c \leq \EU(a \wedge \Box c, b \wedge \Box c)$, so $a \leq d$. To prove that $b \wedge \diam d \leq d$, notice first that
\[ b \wedge c \wedge \diam d  = (b \wedge c \wedge \diam \neg c) \vee (b \wedge c \wedge \diam \EU(a \wedge \Box c,b \wedge \Box c)) = (b \wedge c \wedge \diam \EU(a \wedge \Box c,b \wedge \Box c)), \]
where the last equality holds because the assumption that $b \wedge c \leq \Box c$ gives $b \wedge c \wedge \diam \neg c = \bot$.

Moreover, $b \wedge c \leq b \wedge \Box c$, so by the rule (\ref{EUfix}) we obtain
\[ b \wedge c \wedge \diam \EU(a \wedge \Box c,b \wedge \Box c) \leq b \wedge \Box c \wedge \diam \EU(a \wedge \Box c,b \wedge \Box c)  \leq \EU(a \wedge \Box c, b \wedge \Box c).\]
We have proved that $b \wedge c \wedge \diam d \leq \EU(a \wedge \Box c, b \wedge \Box c)$, so $b \wedge \diam d \leq d$, as required.
\end{proof}

\begin{repproposition}{p:EUcAFcfixpoint}
For any elements $p,q,r$ of a $\CTLf$-algebra $\b A$, we have:
\begin{enumerate}
\item $\EU_c(p,q,r)$ is the least pre-fixpoint of the monotone function $x \mapsto p \vee (q \wedge \diam (r \wedge x))$, and
\item $\AF_c(p,q,r)$ is the least pre-fixpoint of the monotone function $x \mapsto p \vee \Box \AR(q \vee (r \wedge x),p)$.
\end{enumerate}
\end{repproposition}
\begin{proof}
1. Note that
\begin{align*} 
\EU_c(p,q,r) \wedge r &= (p \wedge r) \vee (q \wedge r \wedge \diam \EU(p \wedge r,q \wedge r)) \\
&= \EU(p \wedge r, q \wedge r),
\end{align*}
where the first equality holds by distributivity, and the second equality holds because $\EU(p',q')$ is a fixpoint of $x \mapsto p' \vee (q' \wedge \diam x)$. It follows immediately that $\EU_c(p,q,r)$ is a fixpoint of $x \mapsto p \vee (q \wedge \diam (r \wedge x))$. To prove that it is the least fixpoint, let $s$ be any element of $\b A$ such that $p \vee (q \wedge \diam (r \wedge s)) \leq s$. Then 
\[ (p \wedge r) \vee (q \wedge r \wedge \diam (r \wedge s)) = r \wedge (p \vee (q \wedge \diam (r \wedge s))) \leq r \wedge s.\]
Since $\EU(p',q')$ is a least fixpoint of $x \mapsto p' \vee (q' \wedge \diam x)$, it follows that $\EU(p \wedge r,q \wedge r) \leq r \wedge s$. Hence,
\[\EU_c(p,q,r) = p \vee (q \wedge \diam \EU(p \wedge r,q \wedge r)) \leq p \vee (q \wedge \diam(r \wedge s)) \leq s.\]

2. Write $\sigma$ for the function $x \mapsto p \vee \Box\AR(q \vee (r \wedge x),p)$. 

{\bf -- $\AF_c(p,q,r)$ is a pre-fixpoint of $\sigma$.}\\
Note first, since $\AF_c(p,q,r) \leq \AF(p,q)$, that 
\begin{align}\label{eq:AFcandAF}
p \vee \Box\AR(q \vee \AF_c(p,q,r),p) \leq p \vee \Box \AR(q \vee \AF(p,q),p) = \AF(p,q),
\end{align}
by the fixpoint definition of $\AF$.

Note, using the distributive law and Lemma~\ref{l:basicsyntfacts}.2, that
\begin{align*}
\sigma(\AF_c(p,q,r)) &= (p \vee \Box \AR(q \vee r,p)) \wedge (p \vee \Box \AR(q \vee \AF_c(p,q,r)),p) \\
&\leq (p \vee \Box \AR(q \vee r,p)) \wedge \AF(p,q) = \AF_c(p,q,r),
\end{align*}
where we use (\ref{eq:AFcandAF}) for the inequality from the first to the second line.\\

{\bf -- $\AF_c(p,q,r)$ is less than or equal to any pre-fixpoint of $\sigma$.}\\ 
Let $s \in A$ be such that $\sigma(p,q,r,s) \leq s$.
Write $\alpha := \AR(q \vee r, p)$.

\begin{claimfirst}\label{c:alphaineq}
$\AF(p \vee \diam \neg \alpha,q \vee s) \wedge \Box \alpha \leq s$.
\end{claimfirst}
\begin{pfclaim}~\ref{c:alphaineq}. 
Since $\sigma(p,q,r,s) \leq s$, we have
\begin{equation}\label{eq:xprefix1}
p \leq s,
\end{equation}
and
\begin{equation}\label{eq:xprefix2}
\Box \AR(q \vee (r \wedge s),p) \leq s. 
\end{equation}

Note that 
\begin{align*}
\neg q \wedge \neg s \wedge \alpha &\leq \neg p \wedge \alpha && (\text{by equation (\ref{eq:xprefix1}}))\\
&\leq \Box \alpha && (\text{by Lemma~\ref{l:basicsyntfacts}.5}).
\end{align*}
Therefore, by Lemma~\ref{l:basicsyntfacts}.6, we have
\begin{align}\label{eq:alphainside}
\neg \AR(q \vee s,p) \wedge \alpha = \EU(\neg q \wedge \neg s, \neg p) \wedge \alpha \leq \EU(\neg q \wedge \neg s \wedge \Box \alpha, \neg p \wedge \Box \alpha).
\end{align}
By De Morgan duality and applying $\Box$ on both sides, we obtain from (\ref{eq:alphainside}) that
\begin{align}\label{eq:alphainsidedual}
\Box \AR(q \vee s \vee \diam \neg \alpha, p \vee \diam \neg \alpha) \leq \Box(\AR(q \vee s,p) \vee \neg \alpha).
\end{align}
Since $\neg \alpha = \neg \AR(q \vee r,p)$ by definition, and $\AR(q \vee (r \wedge s),p) = \AR(q \vee r,p) \wedge \AR(q \vee s,p)$ by Lemma~\ref{l:basicsyntfacts}.2, we have
\begin{align}\label{eq:addr}
\Box(\AR(q \vee s,p) \vee \neg \alpha) = \Box(\AR(q \vee (r \wedge s),p) \vee \neg \alpha).
\end{align}
In any modal algebra, $\Box(a \vee b) \leq \Box a \vee \dia b$, so combining (\ref{eq:alphainsidedual}) and (\ref{eq:addr}), we obtain
\begin{align}
\Box \AR(q \vee s \vee \diam \neg \alpha, p \vee \diam \neg \alpha) &\leq \Box \AR(q \vee (r \wedge s),p) \vee \diam \neg \alpha \nonumber\\
&\leq s \vee \diam \neg \alpha &&(\text{by (\ref{eq:xprefix2})}).\label{eq:BoxAR}
\end{align}
From (\ref{eq:BoxAR}) and (\ref{eq:xprefix1}), we conclude that
\[ (p \vee \diam \neg \alpha) \vee \Box \AR(q \vee s \vee \diam \neg \alpha, p \vee \diam \neg \alpha) \leq s \vee \diam \neg \alpha.\]
The rule (\ref{AFmin}) now yields
\[ \AF(p \vee \diam \neg \alpha,q \vee s) \leq s \vee \diam \neg \alpha,\]
from which the claim is clear.
\end{pfclaim}
From the definitions of $\AF_c$ and $\alpha$, distributivity, and monotonicity of $\AF$, we obtain 
\begin{align*} 
\AF_c(p,q,r) &= \AF(p,q) \wedge (p \vee \Box \alpha) \\
&= (\AF(p,q) \wedge p) \vee (\AF(p,q) \wedge \Box\alpha) \\
&\leq p \vee (\AF(p \vee \diam \neg \alpha,q \vee s) \wedge \Box \alpha) \leq s,
\end{align*}
where the last inequality holds by (\ref{eq:xprefix1}) and Claim~\ref{c:alphaineq}.
\end{proof}

The following is a general lemma about least fixpoints; this is the version of \cite[Prop. 5.7(vi)]{Koz1983} that we need here.
\begin{lemma}\label{l:gencontrule}
Suppose that $S$ and $\sigma$ are $(n+1)$-ary operations on a Boolean algebra $\b{B}$ such that, for all $\overline{p} \in B^n$ and $r \in B$, 
\[S(\overline{p},r) \text{ is the least fixpoint of } x \mapsto \sigma(\overline{p},r \wedge x). \]
Then, for any $\overline{p} \in B^n$, $r, \gamma \in B$, 
\[ \text{if } \gamma \wedge S(\overline{p},r) \neq \bot, \text{ then } \gamma \wedge S(\overline{p},r \wedge \neg \gamma) \neq \bot.\]
\end{lemma}
\begin{proof}
Let $\overline{p} \in B^n$, $r, \gamma \in B$. Writing $\delta := \neg \gamma$, we may prove the contrapositive statement, which says that if $S(\overline{p},r \wedge \delta) \leq \delta$, then $S(\overline{p},r) \leq \delta$. Suppose that $S(\overline{p},r\wedge \delta) \leq \delta$. Then $S(\overline{p},r \wedge \delta) = \delta \wedge S(\overline{p},r \wedge \delta)$, so
\[ \sigma(\overline{p},r \wedge S(\overline{p},r \wedge \delta)) = \sigma(\overline{p},r \wedge \delta \wedge S(\overline{p},r \wedge \delta)) = S(\overline{p},r \wedge \delta),\]
where we use in the last equality the fact that $S(\overline{p},r \wedge \delta)$ is a fixpoint. Hence, since $S(\overline{p},r)$ is a \emph{least} fixpoint, we obtain $S(\overline{p},r) \leq S(\overline{p}, r \wedge \delta)$. Since $S(\overline{p}, r \wedge \delta) \leq \delta$ by assumption, we conclude that $S(\overline{p},r) \leq \delta$, as required.
\end{proof}

Combining Prop.~\ref{p:EUcAFcfixpoint} and Lemma~\ref{l:gencontrule} now immediately gives:
\begin{repproposition}{p:contextrules}
For any elements $p,q,r,\gamma$ of a $\CTLf$-algebra $\b{A}$, we have
\begin{enumerate}
\item if $\gamma \wedge \EU(p,q,r) \neq \bot$, then $\gamma \wedge \EU(p,q,r\wedge \neg \gamma) \neq \bot$,
\item if $\gamma \wedge \AF(p,q,r) \neq \bot$, then $\gamma \wedge \AF(p,q, r \wedge \neg \gamma) \neq \bot$.
\end{enumerate}
\end{repproposition}

\subsection{Proofs for Subsection~\ref{ss:preliminaries}}

\begin{replemma}{l:nnf}
Any $\CTLf$-formula is equivalent to a $\CTLf$-formula in negation normal form.
\end{replemma}
\begin{proof}
We first inductively define a `formal negation' $\overline{\phi}$ for any $\CTLf$-formula $\phi$:
\begin{itemize}
\item $\overline{\bot} := \top$,
\item $\overline{p} := \neg p$,
\item $\overline{\neg \phi} := \phi$,
\item $\overline{\phi \vee \psi} := \overline{\phi} \wedge \overline{\psi}$,
\item $\overline{\diam \phi} := \Box \overline{\phi}$,
\item $\overline{\EU(\psi_1,\psi_2)} := \AR(\overline{\psi_1},\overline{\psi_2})$,
\item $\overline{\EG(\psi_1,\psi_2)} := \AF(\overline{\psi_1},\overline{\psi_2})$.
\end{itemize}
Clearly, for any $\CTLf$-formula $\psi$, the formula $\neg \psi$ is equivalent to $\overline{\psi}$. Thus, given an arbitrary $\CTLf$-formula $\phi$, we may replace all negations that occur in $\phi$ by formal negations, after which we obtain an equivalent formula in negation normal form.
\end{proof}
Note that in this proof, we only needed binary operations $\EU$ and $\AF$ in the negation normal form. However, later in the completeness proof, the ternary operations $\EU$ and $\AF$ will come up, which is why we included them in the definition of negation normal form anyway.

\begin{replemma}{l:diamstep}
Let $\b{A}$ be a modal algebra with dual frame $\b{A}_*$. If $a \in A$, $x \in \b{A}_*$, and $\diam a \in x$, then there exists $y \in \b{A}_*$ such that $x {R_*} y$ and $a \in y$.
\end{replemma}
\begin{proof}
Note that the set $\{b \in A \ | \ \diam b \not\in x\}$ is an ideal in $\b{A}$ which does not contain $a$. By the ultrafilter principle, choose an ultrafilter $y$ containing $a$ which is disjoint from this ideal. By construction, $x {R_*} y$.
\end{proof}

\begin{replemma}{l:charform}
For any set of formulas $\rho$ and points $x, x' \in \b{A}_*$, we have
\[ x \sim_\rho x' \iff \kappa(x,\rho) \in x'.\]
\end{replemma}
\begin{proof}
Since $x'$ is an ultrafilter, we have $\kappa(x,\rho) \in x'$ iff for all $\gamma \in x \cap \rho$, $\gamma \in x'$, and for all $\gamma \in \rho \setminus x$, $\gamma \not\in x'$. The latter says precisely that $x \cap \rho = x' \cap \rho$.
\end{proof}

\begin{replemma}{l:jump}
Let $\rho$ be a finite set of formulas, let $\hs \in \{\EU,\AF\}$, and let $\phi$, $\psi$, and $\chi$ be formulas. For any $x \in \b{A}_*$ such that $\hs(\phi,\psi,\chi) \in x$, there exists $x' \in \b{A}_*$ such that $x \sim_\rho x'$ and $\hs(\phi,\psi,\chi \wedge \neg\kappa(x,\rho)) \in x'$.
\end{replemma}
\begin{proof}
Since $\kappa(x,\rho) \wedge \hs(\phi,\psi,\chi) \in x$, we have $\kappa(x,\rho) \wedge \hs(\phi,\psi,\chi) \neq \bot$. By Proposition~\ref{p:contextrules}, $\kappa(x,\rho) \wedge \hs(\phi,\psi,\chi \wedge \neg \kappa(x,\rho)) \neq \bot$. By the Stone ultrafilter theorem, pick $x'$ such that $\kappa(x,\rho) \wedge \hs(\phi,\psi,\chi \wedge \neg \kappa(x,\rho)) \in x'$. By Lemma~\ref{l:charform}, $x \sim_\rho x'$.
\end{proof}

\begin{replemma}{l:finiteclosure}
The closure of a finite set of $\CTLf$-formulas is finite.
\end{replemma}
\begin{proof}
Given a finite set of $\CTLf$ formulas, $\Gamma$, define the set $\Gamma'$ obtained by first applying each of the rules of Definition~\ref{d:closure} to elements of $\Gamma$, and then adding all subformulas. The set $\Gamma'$ is easily seen to be closed and finite.
\end{proof}

\subsection{Proofs for Subsection~\ref{ss:model}}

\begin{lemma} \label{l:AFunravel}
For any elements $p,q,r$ of a $\CTLf$-algebra $\b{A}$, we have
\[\AF_c(p,q,r) = p \vee \Box((q \vee r) \wedge \AF_c(p,q,r)).\]
\end{lemma}
\begin{proof}
Write $t := q \vee (r \wedge \AF_c(p,q,r))$. By Proposition~\ref{p:EUcAFcfixpoint}, $\AF_c(p,q,r) = p \vee \Box \AR(t,p)$. Using this fact and (\ref{ARfix}), we have
\[ \AR(t,p) = t \wedge (p \vee \Box\AR(t,p)) = t \wedge \AF_c(p,q,r) = (q \vee r) \wedge \AF_c(p,q,r).\]
Therefore,
\[ \AF_c(p,q,r) = p \vee \Box \AR(t,p) = p \vee \Box((q \vee r) \wedge \AF_c(p,q,r)).\vspace{-5mm}\]
\end{proof}

\begin{replemma}{l:wellformed}
The one-step unravelling of a well-formed partial tableau $T$ is well-formed. 
\end{replemma}
\begin{proof}
All conditions in Definition~\ref{d:wellformed} except for (g) follow immediately from the definitions. Let $v$ be a leaf of $T$, $m$ the active index at $v$, and $w_\lambda$ a child of $v$ in the one-step unravelling of $T$. Let $1 \leq k \leq |\ell(w_\lambda)|$ be such that $\sigma(w_\lambda)_k \neq \e$. We need to show that $\hs_k(\phi_k,\psi_k,\chip(w_\lambda)_k) \in \alpha(w_\lambda)$. We distinguish three cases:
\begin{itemize}
\item {\bf Case $k < m$.} By definition of $m$, we have $\sigma(v)_k \neq \a$, but $\sigma(w_\lambda)_k \in \{\a,\f\}$. This can only happen when $\sigma(v)_k = \f$, so in particular $\hs_k = \AF$. By well-formedness of $T$, we have $\AF(\phi_k,\psi_k,\chip(v)_k) \in \rho_m \cap \alpha(v)$, so $\AF(\phi_k,\psi_k,\chip(v)_k) \in x_v$. Also, since $\sigma(v)_k \neq \e$, we have $\neg \phi_k \in x_v$. From Lemma~\ref{l:AFunravel}, it follows that $\AF(p,q,r) \wedge \neg p \leq \Box \AF(p,q,r)$, so we obtain $\AF(\phi_k,\psi_k,\chip(v)_k) \in \alpha(w_\lambda)$. This is enough, since $\chip(w_\lambda)_k = \chip(v)_k$ by definition.
\item {\bf Case $k = m$.} If $\hs_m = \EU$, then $\sigma(w_\lambda)_m \neq \e$ only if $\lambda = \chi_m \wedge \EU(\phi_m,\psi_m,\chi_m)$, and it is true by construction that $\EU(\phi_m,\psi_m,\chip(w_\lambda)_m) \in \alpha(w_\lambda)$. If $\hs_m = \AF$, note that $\AF(\phi_m,\psi_m,\chip(w_\lambda)_m) \in x_v$ by construction. Also, $\neg \phi_m \in x_v$ since $\sigma(v)_m = \a$. Again using Lemma~\ref{l:AFunravel}, $\AF(\phi_m,\psi_m,\chip(w_\lambda)_m) \in \alpha(w_\lambda)$.
\item {\bf Case $k > m$.} If $\hs_k = \EU$, then, since $\sigma(w_\lambda)_k \neq \e$, by rule (5) in Definition~\ref{d:onestepunravel}, we must have $\lambda = \chi_k \wedge \EU(\phi_k,\psi_k,\chi_k)$. In particular, $\EU(\phi_k,\psi_k,\chi_k) \in \alpha(w_\lambda)$. 
If $\hs_k = \AF$, then, since $\chip(v)_k \leq \chi_k$ and $\AF(\phi_k,\psi_k,\chip(v)_k) \in \alpha(v)$ by well-formedness of $T$, we have $\AF(\phi_k,\psi_k,\chi_k) \in \alpha(v)$. Since $\sigma(v)_k \neq \e$, we have $\neg \phi_k \in \alpha(v)$. Since $\AF(\phi_k,\psi_k,\chi_k)$ and $\neg \phi_k$ lie in $\Gamma_0 \subseteq \rho(v)_k$, we also have $\AF(\phi_k,\psi_k,\chi_k)$ and $\neg \phi_k$ in $x_v$. As before, $\Box\AF(\phi_k,\psi_k,\chi_k) \in x_v$, so $\AF(\phi_k,\psi_k,\chi_k) \in \alpha(w_\lambda)$.
\end{itemize}
\vspace{-5mm}
\end{proof}

\begin{replemma}{l:truthlemma}[Truth Lemma]
For all $\theta \in \Gamma_0$ and $v \in S$, if $\theta \in \alpha(v)$, then $v \forces \theta$.
\end{replemma}

\begin{proof}
By induction on the complexity of $\theta$. The base cases $\theta = p$ and $\theta = \neg p$ are clear, and the cases $\theta = \theta_1 \vee \theta_2$ and $\theta = \theta_1 \wedge \theta_2$ are immediate from the induction hypothesis.

\vspace{3mm}

$\theta = \diam \lambda$. Suppose that $\diam \lambda \in \alpha(v)$. Since $\diam \lambda \in \Gamma_0$, by construction $v$ has a child $w_\lambda$ with $\lambda \in \alpha(w)$. The induction hypothesis gives $w_\lambda \forces \lambda$.

\vspace{3mm}

$\theta = \Box \lambda$. If there is a successor $w$ of $v$ such that $w \forces \neg \lambda$, then by the induction hypothesis applied to $\lambda \in \Gamma_0$, we have $\lambda \not\in \alpha(w)$, so $\neg \lambda \in \alpha(w)$ since $\alpha(w)$ is an ultrafilter. Since $\alpha(w)$ is an $R_*$-successor of $x_v$, we get $\diam \neg\lambda \in x_v$, so $\Box \lambda = \neg\diam\neg\lambda \not\in x_v$ since $x_v$ is an ultrafilter. Since $x_v \sim_{\Gamma_0} \alpha(v)$, we have $\Box \lambda \not\in \alpha(v)$.

\vspace{3mm}

$\theta = \EU(\phi,\psi,\chi)$. Suppose that $\EU(\phi,\psi,\chi) \in \alpha(v)$. We need to show that $v \forces \EU(\phi,\psi,\chi)$, i.e., that $v \forces \phi \vee (\psi \wedge \diam \EU(\phi \wedge \chi,\psi \wedge \chi))$.
If $\phi \in \alpha(v)$, then $v \forces \phi$ by the induction hypothesis and we are done. Otherwise, we have $\neg \phi \in \alpha(v)$, and since $\EU(\phi,\psi,\chi) \wedge \neg \phi \leq \psi \wedge \diam (\chi \wedge \EU(\phi,\psi,\chi))$ by Proposition~\ref{p:EUcAFcfixpoint}, we have $\psi \wedge \diam (\chi \wedge \EU(\phi,\psi,\chi)) \in \alpha(v)$. In particular, $\psi \in \alpha(v)$, and $v \forces \psi$ by the induction hypothesis. Also, since $\diam(\chi \wedge \EU(\phi,\psi,\chi)) \in \alpha(v) \cap \Gamma_0$, there is a child $v_0$ of $v$ such that $\chi \wedge \EU(\phi,\psi,\chi) \in \alpha(v_0)$. By the induction hypothesis, $v_0 \forces \chi$. We will show that $v_0 \forces \EU(\phi \wedge \chi,\psi \wedge \chi)$, by exhibiting a finite path $v_0, \dots, v_\ell$ such that $v_t \forces \psi \wedge \chi$ for all $t < \ell$ and $v_\ell \forces \phi \wedge \chi$.

If $\phi \in \alpha(v_0)$, we are done immediately since then $v_0 \forces \phi$ by the induction hypothesis, and we already saw that $v_0 \forces \chi$. 
Assume $\phi \not\in \alpha(v_0)$, so $\neg \phi \in \alpha(v_0)$.  By rule (1) in Definition~\ref{d:onestepunravel}, since $\EU(\phi,\psi,\chi) \in \alpha(v_0)$, there exists $1 \leq k \leq \ell(v)$ such that $\theta(v_0)_k = \EU(\phi,\psi,\chi)$ and $\sigma(v_0)_k = \a$.\footnote{If this value has been changed to $\e$ by chance because of
rule (5) in Definition~\ref{d:onestepunravel}, then since $\diam(\chi
\wedge \EU(\phi,\psi,\chi)) \in \alpha(v_0) \cap \Gamma_0$, we can
replace $v_0$ by its  successor $w_{\lambda}\in C_{v_0}$, for 
$\lambda:=\chi \wedge \EU(\phi,\psi,\chi)$: to this node $w_{\lambda}$,
rule (5) for $\lambda$ does not apply and so, if we still have that
$\neg \phi \in \alpha(w_{\lambda})$, then there exists $1 \leq k \leq
\ell(v_0)$ such that $\theta(w_{\lambda})_k = \EU(\phi,\psi,\chi)$ and
$\sigma(w_{\lambda})_k = \a$.
}  
Also, since in general $\EU(p,q,r) \wedge \neg p \leq q$, we have $\psi \in \alpha(v_0)$. 

Suppose by induction that we have constructed a finite path $v_0, \dots, v_\ell$ such that $\psi \wedge \chi \wedge \EU(\phi, \psi, \chi) \wedge \neg \phi \in \alpha(v_t)$ and $\sigma(v_t)_k = \a$ for all $t \leq \ell$.
Then $\diam (\chi \wedge \EU(\phi,\psi,\chi)) \in \alpha(v_\ell) \cap \Gamma_0$, so by construction of the tableau, there is a successor $v_{\ell+1} = w_{\chi \wedge \EU(\phi,\psi,\chi)}$ of $v_\ell$ such that $\chi \wedge \EU(\phi,\psi,\chi) \in \alpha(v_{\ell+1})$. If $\phi \in \alpha(v_{\ell+1})$ we are done, otherwise we will have again $\psi \wedge \chi \wedge \EU(\phi,\psi,\chi) \wedge \neg \phi \in \alpha(v_{\ell+1})$ and $\sigma(v_{\ell+1})_k = \a$.
If, by continuing this process, we would never reach a node $v_\ell$ with $\phi \in \alpha(v_\ell)$, we would obtain an infinite path $(v_t)_{t=0}^\infty$ starting in $v_0$ with $\sigma(v_t)_k = \a$ for all $t \geq 0$, which is impossible by Lemma~\ref{l:noloops}.

\vspace{3mm}

$\theta = \EG(\phi,\psi)$. Suppose that $\EG(\phi,\psi) \in \alpha(v)$. We construct an infinite path with $\phi$ holding everywhere and $\psi$ holding infinitely often. Let $v_0 := v$. Since $\EG(\phi,\psi) \in \alpha(v_0)$, we have $\phi \in \alpha(v_0)$, so $v_0 \forces \phi$ by the induction hypothesis. Since $\Gamma_0$ is closed, $\diam \EU(\psi \wedge \EG(\phi,\psi),\phi) \in \alpha(v) \cap \Gamma_0$. By construction, there exists a child $v_1$ of $v$ such that $\EU(\psi \wedge \EG(\phi,\psi),\phi) \in \alpha(v_1)$. By the $\EU$-case, pick a finite path $v_1, \dots, v_{t_1}$ such that $\phi \in \alpha(v_t)$ for all $1 \leq t < t_1$ and $\psi \wedge \EG(\phi,\psi) \in \alpha(v_{t_1})$. By the induction hypothesis, $v_t \forces \phi$ for all $1 \leq t < t_1$ and $v_{t_1} \forces \psi$. Continuing this process, we obtain an infinite path $(v_t)_{t=0}^\infty$ and an infinite sequence $(t_i)_{i=1}^\infty$ such that $v_t \forces \phi$ for all $t$ and $v_{t_i} \forces \psi$ for all $i$.

\vspace{3mm}

$\theta = \AR(\phi,\psi)$. Suppose that $v \not\forces \AR(\phi,\psi)$. Pick a finite $R$-path $v = v_0, \dots, v_\ell$ such that $v_t \forces \neg \psi$ for all $t < \ell$ and $v_\ell \forces \neg \phi$. We prove that $\AR(\phi, \psi) \not\in \alpha(v_t)$ for each $t \in [0,\ell]$. First note that since $v_\ell \not\forces \phi$, the induction hypothesis for $\phi$ gives $\phi \not\in \alpha(v_\ell)$. Since $\AR(\phi,\psi) \leq \phi$, we get $\AR(\phi,\psi) \not\in \alpha(v_\ell)$. Now suppose that $\AR(\phi,\psi) \not\in \alpha(v_t)$ for some $0 < t \leq \ell$. Since $v_{t-1} \not\forces \psi$, we have $\psi \not\in \alpha(v_{t-1})$ by the induction hypothesis on $\psi$. Also, since $x_{v_{t-1}} {R_*}\alpha(v_t)$, we have $\Box\AR(\phi,\psi) \not\in x_{v_{t-1}}$. Since $\Box\AR(\phi,\psi) \in \Gamma_0$ and $x_{v_{t-1}} \sim_{\Gamma_0} \alpha(v_{t-1})$, we get $\Box\AR(\phi,\psi) \not\in \alpha(v_{t-1})$. Therefore, $\psi \vee \Box \AR(\phi,\psi) \not\in \alpha(v_{t-1})$, so $\AR(\phi,\psi) \not\in \alpha(v_{t-1})$. We conclude that $\AR(\phi,\psi) \not \in \alpha(v_0)$, as required.

\vspace{3mm}

$\theta = \AF(\phi,\psi,\chi)$. Suppose that $\AF(\phi,\psi,\chi) \in \alpha(v)$. We need to show that (i) $v \forces \AF(\phi,\psi)$ and (ii) $v \forces \phi \vee \Box \AR(\psi \vee \chi,\phi)$.

(i) Let $(v_t)_{t=0}^\infty$ be an infinite $R$-path with $v_0 = v$. We need to show that either $v_t \forces \phi$ for some $t \geq 0$, or the path is not $\neg\psi$-fair, i.e., there exists $\tilde{t}$ such that $v_t \forces \psi$ for all $t \geq \tilde{t}$.

Suppose that $v_t \not\forces \phi$ for all $t \geq 0$. 
By the induction hypothesis, $\phi \not\in \alpha(v_t)$ for all $t \geq 0$.
By rule (1) in Definition~\ref{d:onestepunravel}, there exists $1 \leq k \leq \ell(v)$ such that $\theta(v)_k = \AF(\phi,\psi,\chi)$. For all $t \geq 0$, since $\phi \not\in \alpha(v_t)$, we have $\sigma(v)_k \neq \e$. By Lemma~\ref{l:noloops}, pick $\tilde{t} \geq 0$ such that, for all $t \geq \tilde{t}$, we have $\sigma(v_{t})_k = \f$, and $\sigma(v_{t})_{k'} \neq \a$ for all $k' < k$. In particular, for all $t \geq \tilde{t}$, we have that $k < m_{v_{t}}$, the active index in the one-step unravelling at $v_{t}$. Hence, for all $t \geq \tilde{t}$, we must have $\psi \in \alpha(v_{t})$, for otherwise we would get $\sigma(v_{t+1}) = \a$ by rule (7) in Definition~\ref{d:onestepunravel}. By the induction hypothesis, $v_t \forces \psi$ for all $t \geq \tilde{t}$.

(ii) Suppose that $v \not\forces \phi$. We need to show that $v \forces \Box\AR(\psi \vee \chi,\phi)$. By the induction hypothesis, $\phi \not\in \alpha(v)$, so $\neg \phi \in \alpha(v)$. Since also $\AF(\phi,\psi,\chi) \in \alpha(v)$, we obtain $\Box \AR(\psi \vee \chi,\phi) \in \alpha(v)$. Let $w$ be a successor of $v$. Then $\AR(\psi \vee \chi,\phi) \in \alpha(w)$, since $\alpha(w)$ is an $R_*$-successor of $x_v$, $x_v \sim_{\Gamma_0} \alpha(v)$, and $\Box \AR(\psi \vee \chi,\phi) \in \Gamma_0$. By the argument from the case $\theta = \AR(\phi,\psi)$ (see above), we get $w \forces \AR(\psi \vee \chi,\phi)$, as required.
\end{proof}

\begin{replemma}{l:noloops}
For all $v \in S$, $1 \leq k \leq \ell(v)$, and infinite $R$-paths $(v_t)_{t=0}^\infty$ with $v_0 = v$, there exists $t \geq 0$ such that 
\begin{itemize}
\item if $\hs(v)_k = \EU$, then $\sigma(v_t)_k = \e$.
\item if $\hs(v)_k = \AF$, then either $\sigma(v_t)_k = \e$, or for all $t' \geq t$, $\sigma(v_{t'})_k = \f$.
\end{itemize}
\end{replemma}
\begin{proof}
\newcommand{\tz}{{\tilde{t}}}
Let $v \in S$ and $(v_t)_{t = 0}^\infty$ an infinite $R$-path with $v_0 = v$. For $t \geq 0$, we define
\[ m_t := \min \{1 \leq k \leq \ell(v_t) \ | \ \sigma(v_t)_k = \a\},\]
i.e., $m_t$ is the active index at $v_t$ in its one-step unravelling.
We first prove the following claim.
\begin{claimfirst}\label{c:finiteAk}
For each $k \in \{1,\dots,\ell(v)\}$, the set 
\[ A_k := \{ u \geq 0 \ | \ \sigma(v_u)_k = \a\}\]
is finite.
\end{claimfirst}

\begin{pfclaim} \ref{c:finiteAk}.
By induction on $k$. Assume $A_{k'}$ is finite for all $1 \leq k' < k$. Choose $\tz \geq 0$ such that $\bigcup_{k'=1}^{k-1} A_{k'}$ is contained in $[0,\tz-1]$. (In particular, if $k = 1$, we may simply choose $\tz = 0$.) Thus, for any $t \geq \tz$, we have $\sigma(v_t)_{k'} \neq \a$ for all $k' < k$, so $m_t \geq k$. Therefore, for all $t \geq \tz$, $\rho(v_t)_{k} = \rho(v_\tz)_{k}$, because in the one-step unravelling of the partial tableau, the relevance set at position $k$ can only be updated when $m_t < k$. Write $\rho := \rho(v_\tz)_{k}$. We now prove the following.

\begin{claim}\label{c:distincttypes}
For any distinct $u, u'$ in $A_k \cap [\tz,\infty)$, $\alpha(v_u)$ and $\alpha(v_{u'})$ have distinct $\rho$-types. 
\end{claim}

\begin{pfclaim} \ref{c:distincttypes}. Let $u, u + d \in A_k$ for some $u \geq \tz$ and $d > 0$. Since $m_u \geq k$ and $\sigma(v_u)_k = \a$, we have $m_u = k$, so $\chip(v_{u+1})_k = \chip(v_u)_k \wedge \neg \gamma_{v_u}$. Also, since $m_{u+t} \geq k$ for all $0 \leq t \leq d$, we have $\chip(v_{u+1})_k \geq \cdots \geq \chip(v_{u+d})_k$.
\begin{itemize}
\item {\bf Case $\hs_k = \EU$.} Since $\sigma(v_{u+d}) = \a$, we must have $\sigma(v_{u + t})_k = \a$ for all $0 \leq t \leq d$. 
Since $m_t \geq k$ if $t \geq \tz$, we get $m_u = m_{u+1} = \cdots = m_{u+d} = k$. By the construction of the one-step unravelling, case $\hs_m = \EU$, we then obtain $\chip(v_{u+t})_k \in \alpha(v_{u+t})$ for all $1 \leq t \leq d$. Moreover, $\neg \gamma_{v_u} \geq \chip(v_{u+1})_k \geq \chip(v_{u+d})_k$, so $\neg \gamma_{v_u} \in \alpha(v_{u+d})$. Since $\gamma_{v_u} = \kappa(u,\rho)$, Lemma~\ref{l:charform} gives that $\alpha(v_{u+d}) \not\sim_{\rho} \alpha(v_u)$.
\item {\bf Case $\hs_k = \AF$.} Let us write $\pi$ for the formula $\AF(\phi_k,\psi_k,\chip(v_u)_k \wedge \neg \gamma_{v_u})$, and $x$ for $x_{v_{u+d-1}}$. We show first that $\pi \in x$, by distinguishing three sub-cases.
\begin{itemize}
\item If $d = 1$, then $\pi \in x$ by construction. 
\item If $d > 1$ and $m_{u+d-1} = k$, so $\sigma(v_{u+d-1})_k = \a$, then by the choice of $x$ we have $\AF(\phi_k,\psi_k,\chip(v_{u+d-1})_k \wedge \neg \gamma_{v_{u+d-1}}) \in x$. Note that $\chip(v_{u+d-1})_k \leq \chip(v_{u+1})_k = \chip(v_u)_k \wedge \neg \gamma_{v_u}$, so $\pi \in x$.
\item If $d > 1$ and $m_{u+d-1} > k$, then $\chip(v_{u+d-1})_k \leq \chip(v_u)_k \wedge \neg \gamma_{v_u}$, and also, by well-formedness, we have $\AF(\phi_k,\psi_k,\chip(v_{u+d-1})_k) \in \alpha(v_{u+d-1})$. In particular, $\pi \in \alpha(v_{u+d-1})$. Note that 
 $\pi \in \rho(v_{u+d-1})_{m_{u+d-1}}$, because, in the one-step unravelling of the node $v_u$, $\pi$ was added to all relevance sets $\rho(v_{u+1})_{k'}$ for $k' > k = m_u$, by rule (3) in Definition~\ref{d:onestepunravel}, and thus $\pi$ also lies in any relevance sets that appeared later, by rule (1). Therefore, since $\alpha(v_{u+d-1})$ and $x$ have the same type with respect to $\rho(v_{u+d-1})_{m_{u+d-1}}$, we obtain $\pi \in x$.
\end{itemize}
Since $\sigma(v_{u+d})_k = \a$, we must have $\neg \phi_k \in \alpha(v_{u+d-1})$, so $\neg \phi_k \in x$. Applying Lemma~\ref{l:AFunravel}, $\pi \wedge \neg \phi_k \leq \Box(\psi_k \vee (\chip(v_u)_k \wedge \neg \gamma_{v_u}))$. In particular, $\psi_k \vee (\chip(v_u)_k \wedge \neg \gamma_{v_u})) \in \alpha(v_{u+d})$, since $\alpha(v_{u+d})$ is an $R_*$-successor of $x$. Now, because $\sigma(v_{u+d})_k \neq \f$, we must have $\psi_k \not\in \alpha(v_{u+d})$ by rule (6) in Definition~\ref{d:onestepunravel}. Therefore, $\chip(v_u)_k \wedge \neg \gamma_{v_u} \in \alpha(v_{u+d})$. In particular, $\neg \gamma_{v_u} \in \alpha(v_{u+d})$, so $\alpha(v_{u+d}) \not\sim_{\rho} \alpha(v_{u})$ by Lemma~\ref{l:charform}.
\end{itemize}
This concludes the proof of Claim~\ref{c:distincttypes}.
\end{pfclaim}

From Claim~\ref{c:distincttypes}, since only $2^{|\rho|}$ $\rho$-types exist, it follows that $|A_k \cap [\tz, \infty)| \leq 2^{|\rho_0|}$. Since $A_k \subseteq [0,\tz] \cup (A_k \cap [\tz,\infty))$, from this we can conclude that $|A_k| \leq \tz + 2^{|\rho_0|}$. This concludes the proof of Claim~\ref{c:finiteAk}.
\end{pfclaim}
By Claim~\ref{c:finiteAk}, define $t_0 := \max{A_k} + 1$. If $\hs_k = \EU$, then $\sigma(v_{t_0})_k \neq \f$, so we must have $\sigma(v_{t_0})_k = \e$, and we can choose $t := t_0$. If $\hs_k = \AF$, then either there exists $t \geq t_0$ such that $\sigma(v_{t}) = \e$, or otherwise $\sigma(v_{t'}) = \f$ for all $t' \geq t_0$, in which case we can choose $t := t_0$.
\end{proof}

\subsection{Proofs for Subsection~\ref{ss:compvariants}}

\begin{reptheorem}{t:rootedcompleteness}
For every consistent $\CTLfp$-formula $\phi_0(\p)$, there exists a $\p$-coloured tree such that for some node $s$, $s \forces \phi_0$.
\end{reptheorem}

\begin{proof}
 First notice that if a $\CTLfp$-formula $\phi$ is consistent, then  $\I\wedge \EU(\phi, \top)\wedge \Box \AR(\neg \I, \bot)$ is also consistent.
 Indeed, interpreting $\phi$ in rooted $\CTLf$-algebras,  from $\phi\neq \bot$ we get $I\leq \EU(\phi, \top)$, so that $\I\wedge \EU(\phi, \top)\wedge \Box \AR(\neg \I, \bot)$ is equal to $\I$, and $\I\neq \bot$ is an axiom.
 
 Now, if we apply the above tableau construction to $\I\wedge \EU(\phi, \top)\wedge \Box \AR(\neg \I, \bot)$, we get a tree model where $\phi$ holds somewhere and $\I$ holds only in the root.
\end{proof}

For the binary case, we indicate the elements of the proof that are different from the case treated in the previous subsection.

The Fischer-Ladner closure of a finite set of formulas is so modified:
\begin{definition}\label{d:b-closure}
A set of $\CTLf$ formulas $\Gamma$ is called (Fischer-Ladner) \emph{closed} if the following hold:
\begin{itemize}
\item $\EU(\top,\top,\top) \in \Gamma$.
\item if $\phi \in \Gamma$, then $\phi' \in \Gamma$ for any subformula $\phi'$ of $\phi$.
\item if $\Diamond \phi \in \Gamma$,
then $\X_0\phi\in \Gamma$ and $\X_1\phi \in \Gamma$.
\item if $\EG(\phi,\psi) \in \Gamma$, then $\diam \EU(\psi \wedge \EG(\phi,\psi),\phi) \in \Gamma$.
\item if $\AR(\phi,\psi) \in \Gamma$, then $\Box \AR (\phi,\psi) \in \Gamma$.
\item if $\EU(\phi,\psi,\chi) \in \Gamma$, then $\diam(\chi \wedge \EU(\phi,\psi,\chi)) \in \Gamma$,
\item if $\AF(\phi,\psi,\chi) \in \Gamma$, then
$\Box \AR(\psi \vee \chi,\phi) \in \Gamma$.
\end{itemize}
The \emph{closure} of a set of $\CTLf$ formulas is the smallest closed set containing it.
\end{definition}

Lemma~\ref{l:finiteclosure} still holds; we can repeat Definition~\ref{d:charform} and prove Lemma~\ref{l:charform}
and Lemma~\ref{l:jump} for binary $\CTLf$-algebras. We can also restate Definition~\ref{d:partialtableau}, with the only obvious modification that the partial tableau is now based on a finite \emph{binary} tree $T$.

Since the axioms for $\X_0, \X_1$ and the axiom $\Diamond \phi = \X_0\phi \vee \X_1 \phi$ are 
in Sahlqvist form, by standard modal logic machinery~\cite{BRV2001}, we have that in the dual spaces of binary $\CTLf$-agebras the operators $\X_0, \X_1$ correspond to unary functions
(to be called $f_0, f_1$) whose union is the relation $R_*$ dual to the modal operator $\diam$. With this information, we can modify Definition~\ref{d:onestepunravel}
as follows:

\begin{definition}\label{d:b-onestepunravel}
We define the \emph{(binary) one-step unravelling} of a well-formed 
partial tableau $(T,\alpha,\beta)$. For each leaf $v$ of $T$, add two  children $v0$ and $v1$ of $v$ as follows.
%
%
We again choose an auxiliary ultrafilter $x_v \in \b{A}_*$.
Let 
$$
C^0_v := \{\lambda \ | \ \X_0 \lambda \in \Gamma_0 \cap \alpha(v)\}, \quad 
C^1_v := \{\lambda \ | \ \X_1 \lambda \in \Gamma_0 \cap \alpha(v)\}
$$
If $\sigma_k \neq \a$ for all $1 \leq k \leq \ell(v)$, define $x_v := \alpha(v)$. Otherwise, put 
\[m := \min \{1 \leq k \leq \ell(v) \ | \ \sigma_k = \a\}.\]
We call $m$ the \emph{active index} at $v$.
By well-formedness, we have $\hs(\phi_m,\psi_m,\chip_m) \in \alpha(v)$. Therefore, by Lemma~\ref{l:jump}, pick $x_v \in \b{A}_*$ such that $\hs_m(\phi_m,\psi_m,\chip_m \wedge \neg \kappa(\alpha(v),\rho_m)) \in x_v$ and $x_v \sim_{\rho_m} \alpha(v)$. 
We let $v0$ be $f_0(x_v)$ and $v1$ be $f_1(x_v)$.

For each $\lambda$ such that $\diam\lambda  \in \Gamma_0 \cap \alpha(v)$, 
by the revised Definition~\ref{d:b-closure} of a closed set, we have that there is $i=1,2$ such that $\X_i \lambda\in C^i_v$ and so  $\lambda\in f_i(vi)$: we call 
$vi$ a \emph{$\lambda$-designated successor} of $v$. 
In case $\hs_m = \EU$, notice the following (write $\gamma_v := \kappa(\alpha(v),\rho_m)$).
Since the partial tableau is well-formed and $\sigma_m = \a$, we have $\phi_m \not\in \alpha(v)$. Since $\alpha(v) \sim_{\rho_m} x_v$ and $\phi_m \in \Gamma_0 \subseteq \rho_m$, we have $\phi_m \not\in x_v$, so $\neg \phi_m \in x_v$. Also, $\EU(\phi_m,\psi_m,\chip_m \wedge \neg \gamma_v) \in x_v$ by construction. 
Applying the general fact (Proposition~\ref{p:EUcAFcfixpoint}) that $\EU(p,q,r) \wedge \neg p \leq \diam (r \wedge \EU(p,q,r))$, we obtain 
$\diam (\chip_m \wedge \neg \gamma_v \wedge \EU(\phi_m,\psi_m,\chip_m \wedge \neg \gamma_v)) \in x_v$; thus for $i=0$ or $i=1$, we have that 
$\chip_m \wedge \neg \gamma_v \wedge \EU(\phi_m,\psi_m,\chip_m \wedge \neg \gamma_v) \in \alpha(vi)$. 
Thus we can assume that if $\lambda= \chi_m\wedge \EU(\phi_m, \psi_m, \chi_m)$, the $\lambda$-designated successor $vi$ of $v$ is such that $\chip_m \wedge \neg \gamma_v \wedge \EU(\phi_m,\psi_m,\chip_m \wedge \neg \gamma_v) \in \alpha(vi)$.

The word $\beta(vi)$ ($i=0,1$) is defined as an update of the word $\beta(v)$, obtained by consecutively applying the following steps:
\begin{enumerate}
\item Let $\mathsf{New}(vi) := \{ \theta \in \alpha(w) \cap \Gamma_0 \cap \mathrm{Ev} \ | \ \forall 1 \leq k \leq \ell(v) : \text{ if } \theta_k = \theta, \text{ then } \sigma_k = \e\}$.
For each $\theta = \hs(\phi,\psi,\chi) \in \mathsf{New}(vi)$, add one letter, $(\theta,\a,\rho',\chi)$, to the end of the word, where $\rho' := \bigcup_{k=1}^{\ell(v)} \rho_k$.
\item For each position $k$, put
\[ \chip(w_\lambda)_k = \begin{cases} \chip(v)_k &\mbox{if } k < m, \\
									  \chip(v)_m \wedge \neg \gamma_v & \mbox{if } k = m, \\
									  \chi(v)_k &\mbox{if } k > m.\end{cases} \]
\item For each position $k > m$, add the formula $\hs_m(\phi_m,\psi_m,\chip(v)_m \wedge \neg \gamma_v)$ to the set $\rho_k$.
\item For each position $k$ such that $\phi_k \in \alpha(vi)$, change $\sigma_k$ into $\e$.
\item For each position $k$, if $\theta_k = \EU(\phi_k,\psi_k,\chi_k)$ 
and $vi$ is not a $\lambda$-designated successor of $v$ (for $\lambda=\chi_k\wedge\EU(\phi_k,\psi_k,\chi_k)$),
change $\sigma_k$ into $\e$. If, after this operation, it turns out that $\theta_k\wedge \neg\phi_k\in \alpha(w_{\lambda})$, then $\theta_k$ must be treated as a new eventuality, 
so that (as in item 1 above) $(\theta_k,\a,\rho',\chi_k)$ is appended  to the end of the word (where $\rho' := \bigcup_{s=1}^{\ell(v)} \rho_s
\cup\{\hs_m(\phi_m,\psi_m,\chip(v)_m \wedge \neg \gamma_v)\}$).
\item For each position $k$, if $\hs_k = \AF$, $\psi_k \in \alpha(vi)$, and $\sigma_k = \a$, change $\sigma_k$ into $\f$.
\item For each position $k < m$, if $\hs_k = \AF$, $\sigma_k = \f$, $\phi_k \not\in \alpha(vi)$ and $\psi_k \not\in \alpha(w_\lambda)$, change $\sigma_k$ into $\a$.
\end{enumerate}
\vspace{-5mm}
\end{definition}

Lemma~\ref{l:wellformed} still holds; Definition~\ref{d:modelconstruction} can be restated word for word and Lemmas~\ref{l:truthlemma}~and~\ref{l:noloops} are proved as before. Thus, any consistent formula of $\CTLf$ enriched with $X_0$ and $X_1$ is satisfied by some colouring of the full binary tree. Now the same proof as in Theorem~\ref{t:rootedcompleteness} can be used to prove Theorem~\ref{t:binarycompleteness}.

\section{Proofs for Section~\ref{s:modelcompanions}}

\begin{replemma}{l:justequations}
For any quantifier-free $\mathcal{L}$-formula $\phi(\overline{p})$, there exists an $\mathcal{L}$-term $t_\phi(\overline{p})$ such that $\mathrm{CTL}^f_\mathrm{I} \proves \phi \leftrightarrow (t_\phi = \top)$; similarly, there exists an $\mathcal{L}$-term $t_\phi'(\overline{p})$ such that $\mathrm{CTL}^f_\mathrm{I} \proves \phi \leftrightarrow (t_\phi' \neq \bot)$.
\end{replemma}
\begin{proof}
We first construct the term $t_\phi$ by induction on the complexity of $\phi$, which, we may assume, is built up from equalities of $\mathcal{L}$-terms by consecutively applying the Boolean connectives $\wedge$ and $\neg$ from the first-order language. 

If $\phi$ is an atomic formula, then it has the form $t_1 = t_2$ for $\mathcal{L}$-terms $t_1$ and $t_2$, and we may define $t_\phi := (t_1 \wedge t_2) \vee (\neg t_1 \wedge \neg t_2)$.

If $\phi$ is of the form `$\phi_1 \wedge \phi_2$', we can clearly put $t_\phi := t_{\phi_1} \wedge t_{\phi_2}$, where $t_{\phi_1}$ and $t_{\phi_2}$ are defined by induction. Here we use that, for any elements $a$, $b$ in a Boolean algebra, $a = \top$ and $b = \top$ if, and only if, $a \wedge b = \top$.

For the case of negation, notice first that, for any element $a$ in a fair CTL algebra $\b{A}$, we have that
\begin{equation} \label{eq:elimnegation}
a \neq \top \text{ if, and only if, } \I \leq \EU(\neg a, \top).
\end{equation}
Indeed, one direction follows from the last axiom for $\I$, and the other direction follows from the first axiom for $\I$ and the fact that $\EU(\bot,\top) = \bot$, which easily follows from the fixpoint axiom for $\EU$.

Now, if $\phi$ is of the form `$\neg \psi$', then by induction $\phi$ is $\CTLf$-equivalent to $t_\psi \neq \top$. It therefore suffices by the equivalence in (\ref{eq:elimnegation}) to define $t_\phi := \neg I \vee \EU(\neg t_\psi,\top)$.

Now that we have successfully defined $t_\phi$ for all $\mathcal{L}$-formulas $\phi$, we may put $t'_\phi := I \wedge \neg\EU(\neg t_\phi,\top)$. Then $t'_\phi \neq \bot$ iff $\neg I \vee \EU(\neg t_\phi,\top) \neq \top$ iff $I \not\leq \EU(\neg t_\phi,\top)$, which, by (\ref{eq:elimnegation}), is equivalent to $t_\phi = \top$, and the latter is equivalent to $\phi$.
\end{proof}

\subsection{Proofs for Subsection~\ref{ss:exform}}
\begin{lemma}\label{l:omegaunravel}
Let $(S,\sigma)$ be a $\p$-coloured tree with root $s_0$. Define the function $z : S_\omega \to S$ by $z(\epsilon) := s_0$ and $z((k_1,s_1) \dots (k_n,s_n)) := s_n$. 
Then $z$ is a surjective p-morphism.
\end{lemma}
\begin{proof}
By definition of $\sigma_\omega$, we have $v \in \sigma_\omega(p)$ if, and only if, $z(v) \in \sigma(p)$. If $v{R_\omega}v'$ in $S_\omega$, then by definition $z(v){R}z(v')$. If $v \in S_\omega$ has length $n \geq 0$ and $z(v){R}s_{n+1}$, then $v' := v(0,s_{n+1})$ is an element of $S_\omega$ such that $z(v') = s_{n+1}$ and $v{R_\omega}v'$. Finally, $z$ is surjective because for any node $s \in S$, there exists a path from the root of $S$ to $s$, $s_0{R}\dots{R}s_n = s$, so that $v_s := (0,s_1)\dots(0,s_n)$ is an element of $S_\omega$ with $z(v_s) = s$.
\end{proof}

\begin{repproposition}{p:subalgebra}
For any $\p$-coloured tree $(S,\sigma)$, the algebra $\b{P}(S)$ is isomorphic to a subalgebra of $\b{P}(S_\omega)$, via an isomorphism which in particular sends $V_\sigma(p)$ to $V_{\sigma_\omega}(p)$ for each $p$ in $\p$.
\end{repproposition}
\begin{proof}
Let $i : \b{P}(S) \to \b{P}(S_\omega)$ be the function given by $i(a) := z^{-1}(a)$, where $z$ is the surjective p-morphism from Lemma~\ref{l:omegaunravel}. Since $z$ is surjective, $i$ is injective, and it is obviously a homomorphism of Boolean algebras. It is straightforward to check directly that $i$ preserves the operators $\I$, $\diam$, $\EU$ and $\EG$, or, alternatively, one may refer to the general fact that the inverse image map of a p-morphism preserves any operators that are definable in the modal $\mu$-calculus, because modal $\mu$-formulas are bisimulation-invariant. Therefore, the algebra $\b{P}(S)$ is isomorphic to its image under $i$, which is a subalgebra of $\b{P}(S_\omega)$.
\end{proof}

\begin{repproposition}{p:termtoMSO}
For any first-order $\mathcal{L}$-formula $\phi(\p)$, there exists a monadic second order formula $\Phi(\p)$ such that, for any $\p$-coloured tree $(S,\sigma)$, 
\[ \b{P}(S),V_\sigma \models_{\FO} \phi(\p) \iff S, \sigma \models_{\MSO} \Phi(\p).\]
\end{repproposition}
\begin{proof}
Recall that in the proof sketch in the paper, the formula $\Phi(\p)$ has been defined from $\phi$ by replacing each atomic formula $t_1 = t_2$ by the formula $\forall v (\dot{t_1}(\p,v) \leftrightarrow \dot{t_2}(\p,v))$, where $\dot{t_k}$ is the term defined inductively in the proof sketch. It remains to check that this $\Phi(\p)$ satisfies the stated property. The only non-trivial step is that of atomic formulas. For this, the crucial observation is that, for any $\mathcal{L}$-term, $\p$-coloured tree $(S,\sigma)$ and any node $w \in S$, we have
\begin{equation}\label{eq:termeq}
w \in t^{\mathbb{P}(S)} \text{ if, and only if, } S, \sigma[v \mapsto w] \models_{\MSO} \dot{t},
\end{equation}
where $\sigma[v \mapsto w]$ is the extension of $\sigma$ by making the first-order variable $v$ true in the node $w$.
The equivalence (\ref{eq:termeq}) is proved by an induction on the complexity of the term $t$, using the definition of the operations on the complex algebra $\b{P}(S)$ and the definition of $\dot{t}$.
It follows immediately from (\ref{eq:termeq}) that indeed
\[ \b{P}(S),V_\sigma \models_{\FO} t_1 = t_2 \iff S,\sigma \models_{\MSO} \forall v (\dot{t_1}(\p,v) \leftrightarrow \dot{t_2}(\p,v)),\]
as required.
\end{proof}

\begin{repproposition}{p:auttoterm}
For any non-deterministic modal automaton $\mathcal{A}$ over $\p$ with set of states $\q$, there exists an $\mathcal{L}$-term $\acc_{\mathcal{A}}(\p,\q)$ such that for any $\p$-coloured tree $(S,\sigma)$, we have
\[ \mathcal{A} \text{ accepts } (S_\omega,\sigma_\omega) \iff \b{P}(S_\omega),V_{\sigma_\omega} \models \exists \q \; \acc_{\mathcal{A}}(\p,\q) = \top. \]
\end{repproposition}
\begin{proof}
Let $\mathcal{A}$ be a non-deterministic modal automaton over $\p$. 
\begin{claimfirst}\label{c:acceptance}
For any $\p$-coloured tree $(S,\sigma)$, 
there is a bijection between successful runs $r$ of $\mathcal{A}$ on $(S_\omega,\sigma_\omega)$ and valuations $V_r : \q \to \mathcal{P}(S_\omega)$ that satisfy the following three properties:
\begin{enumerate}
\item (Initial) $\epsilon \in V_r(q_0)$;
\item (Transition) for all $v \in S_\omega$, there is a unique $q \in \q$ such that $v \in V_r(q)$, and moreover, for this $q$, the set $\{q' \ | \ v' \in V_r(q') \text{ for some $R$-successor } v' \text{ of } v\}$ is in $\delta(q,\sigma_\omega(v))$;
\item (Success) for all odd $n \in \mathrm{range}(\Omega)$ and for any infinite path $(v_t)_{t\in\omega}$ in the tree such that $v_t \in \bigcup_{\Omega(q) = n} V_r(q)$ for infinitely many $t$, there exists $q' \in \q$ such that $\Omega(q') < n$ and $v_t \in V_r(q)$ for some $t$.
\end{enumerate}
\end{claimfirst}
\begin{pfclaim}~\ref{c:acceptance}. 
The claimed bijection is a restriction of the bijection between $\q$-colourings $r : S_\omega \to \mathcal{P}(\q)$ and valuations $\q \to \mathcal{P}(S_\omega)$. Indeed, for any function $r : S_\omega \to \q$, define $V_r(q) := r^{-1}(q)$ for each $q \in \q$. It is straight-forward to check that $V_r$ verifies conditions (1) and (2) in the Claim if, and only if, $r$ verifies conditions (1) and (2) in the definition of a successful run (Def.~\ref{d:automaton}). 

Regarding condition (3), suppose first that $r$ satisfies (3) in Definition~\ref{d:automaton}. If $(v_t)_{t\in\omega}$ is an infinite path, $n$ is odd and $v_t \in \bigcup_{\Omega(q) = n} V_r(q)$ for infinitely many $t$, then by the pigeon-hole principle there is some $q$ with $\Omega(q) = n$ and $v_t \in V_r(q)$ for infinitely many $t$. Denote by $\rho = w_0,\dots,w_m = v_0$ the unique path from the root $\rho$ of $S_\omega$ to $v_0$, and extend this to an infinite path by defining $w_{m+t} := v_t$. Since $r$ satisfies (3) in Definition~\ref{d:automaton}, there must exist a state $q$ with $\Omega(q) < n$ and $w_t \in V_r(q)$ for infinitely many $t$. In particular, choosing a $t' \geq m$ with $w_{t'} \in V_r(q)$, we see that $v_{t'-m} \in V_r(q)$. Thus, $V_r$ satisfies (3) in the Claim. Conversely, it is clear that if $V_r$ satisfies (3) in the Claim, then $r$ must satisfy (3) in Definition~\ref{d:automaton}.
\end{pfclaim}
Recall the terms $\acc_1,\acc_2,\acc_3$ defined in the proof sketch in the paper. Note that, for $j = 1,2,3$, we have $\acc_j(\p,\q) = \top$ under a valuation $V_{\sigma_\omega} \cup V_r : \p \cup \q \to \mathcal{P}(S_\omega)$ if, and only if, condition ($j$) in Claim~\ref{c:acceptance} holds. Therefore, putting $\acc_{\mathcal{A}}(\p,\q) := \acc_1 \wedge \acc_2 \wedge \acc_3$ gives the required $\mathcal{L}$-term.
\end{proof}

\subsection{Proofs for Subsection~\ref{ss:Tstar}}

\begin{proposition}\label{p:concludeequiv}
For all $j = (t,\p,\x) \in J$ and for any $\p$-coloured tree $(S,\sigma)$, we have
\[ \b{P}(S_\omega), V_{\sigma_\omega} \models \phi_j \leftrightarrow \psi_j.\]
\end{proposition}
\begin{proof}
We have
\begin{align*}
\b{P}(S_\omega), V_{\sigma_\omega} \models \phi_j &\iff S_\omega, \sigma_\omega \models \Phi_j &&(\text{Prop.~\ref{p:termtoMSO}}) \\
&\iff \mathcal{A}_{\Phi_j} \text{ accepts } (S_\omega,\sigma_\omega) &&(\text{Prop.~\ref{p:MSOtoaut}}) \\
&\iff \b{P}(S_\omega), V_{\sigma_\omega} \models \psi_j &&(\text{Prop.~\ref{p:auttoterm}}).
\end{align*}
\end{proof}

\begin{reptheorem}{t:main1}
$(\CTLfp)^*$ is the model companion of $\CTLfp$.
\end{reptheorem}
\begin{proof}
We prove that $(\CTLfp)^*$ is a model-complete co-theory of $\CTLfp$.\\

1. $(\CTLfp)^*$ {\bf is model-complete.}

It suffices to prove, for each $j = (t,\p,\x) \in J$, that in all rooted $\CTLf$-algebras,
\begin{equation} \label{eq:converseinT} 
\forall \p ( \psi_j \rightarrow \phi_j ).
\end{equation}
Indeed, given this fact, from $\CTLfp \proves (\ref{eq:converseinT})$, it will follow from the definition of $(\CTLfp)^*$ that every universal formula is equivalent over $(\CTLfp)^*$ to an existential one, so that $(\CTLfp)^*$ is model complete.

We first prove that (\ref{eq:converseinT}) is true in every rooted $\CTLf$-algebra of the form $\b{P}(S)$, where $S$ is a tree. Let $(S,\sigma)$ be any $\p$-coloured tree, and suppose that $\b{P}(S),V_\sigma \models \psi_j(\p)$. Since $(\b{P}(S),V_\sigma)$ embeds into $(\b{P}(S_\omega),V_{\sigma_\omega})$ by Proposition~\ref{p:subalgebra}, and $\psi_j$ is existential, we also have $\b{P}(S_\omega),V_{\sigma_\omega} \models \psi_j(\p)$. By Proposition~\ref{p:concludeequiv}, we obtain $\b{P}(S_\omega),V_{\sigma_\omega} \models \phi_j(\p)$. Since $\phi_j$ is universal and again $(\b{P}(S),V_\sigma)$ is a subalgebra of $(\b{P}(S_\omega),V_{\sigma_\omega})$, we conclude that $\b{P}(S),V_{\sigma} \models \phi_j(\p)$.


Note that by first-order logic the sentence (\ref{eq:converseinT}) is equivalent to the universal sentence 
\begin{equation} \label{eq:converseinTuniversal}
\forall \p, \q, \x ( \mathsf{acc}_{\mathcal{A}_{t,\x}}(\p,\q) = \top \rightarrow t(\p,\x) = \top).
\end{equation}
By Lemma~\ref{l:justequations}, pick a term $t'(\p,\q,\x)$ such that
\begin{equation}\label{eq:tprimedef}
T \proves (t'(\p,\q,\x) = \top) \leftrightarrow (\mathsf{acc}_{\mathcal{A}_{t,\x}}(\p,\q) = \top \rightarrow t(\p,\x) = \top).
\end{equation}

Since we established above that (\ref{eq:converseinT}) holds in every rooted $\CTLf$-algebra of the form $\b{P}(S)$, where $S$ is a tree, the equation $t'(\p,\q,\x) = \top$ is also valid in every such rooted $\CTLf$-algebra. Therefore, by Theorem~\ref{t:completeness}, the equation $t'(\p,\q,\x) = \top$ is valid in all rooted $\CTLf$-algebras. Hence, (\ref{eq:converseinTuniversal}) holds in all rooted $\CTLf$-algebras, and thus also (\ref{eq:converseinT}), as required.\\

2. $(\CTLfp)^*$ {\bf is a co-theory of }$\CTLfp${\bf{.}}

By \cite[Lem.~3.5.7]{CK}, every $\CTLfp$-algebra embeds into an existentially closed $\CTLfp$-algebra. Therefore, to prove that $(\CTLfp)^*$ is a co-theory of $\CTLfp$, it suffices to prove that every existentially closed $\CTLfp$-algebra is a model of $(\CTLfp)^*$. 

Let $\b{A}$ be an existentially closed $\CTLfp$-algebra and let $j = (t,\p,\x) \in J$ and $\a \in A^n$ be arbitrary. By Lemma~\ref{l:extendalgebra} (proved below), there is an extension of $\b{A}$ where $\phi_j(\oa) \to \psi_j(\oa)$ holds. Note that $\phi_j(\oa) \to \psi_j(\oa)$ is (by first-order logic) an existential sentence in the language $\mathcal{L}_A$. Thus, since $\b{A}$ is existentially closed, $\phi_j(\oa) \to \psi_j(\oa)$ holds in $\b{A}$.
\end{proof}
Note that, in fact, the above proof also shows immediately that the models of $(\CTLfp)^*$ are exactly the existentially (= algebraically) closed models for $\CTLfp$.

\begin{replemma}{l:extendalgebra}
Let $j = (t,\p,\x) \in J$, with $\p = p_1,\dots,p_n$. For any rooted $\CTLf$-algebra $\b{A}$ and $\oa \in A^n$, there is a rooted $\CTLf$-algebra $\b{A}'$ which contains $\b{A}$ as a subalgebra such that $\b A' \models \phi_j(\oa) \to \psi_j(\oa)$.
\end{replemma}

\begin{proof}
Let $\b{A}$ be a rooted $\CTLf$-algebra and $\oa \in A^n$. Consider the language $\mathcal{L}_A := \mathcal{L} \cup \{c_a \ | \ a \in A\}$, where each $c_a$ is a fresh constant symbol. Note that it suffices to prove that the $\mathcal{L}_A$-theory
\[ T' := \CTLfp \cup \{t(\oa,\ob) \neq \bot \ : \ \b{A} \models t(\oa,\ob) \neq \bot\} \cup \{\phi_j(\oa) \to \psi_j(\oa)\}\]
is consistent. Indeed, any model $\b A'$ of the theory $T'$ will contain a subalgebra isomorphic to $\b A$, since any quantifier-free $\mathcal{L}_A$-formula is equivalent to an $\mathcal{L}_A$-formula of the form $t(\oa,\ob) \neq \bot$ by Lemma~\ref{l:justequations}.

In order to prove that $T'$ is consistent, by the compactness theorem of first-order logic, it suffices to prove that every finite subset $U$ of $T'$ is consistent. The crucial step is the following claim.

\begin{claimfirst}\label{c:consistentsubset}
For every $\mathcal{L}$-term $t(\p,\y)$ and tuple $\ob \in A^{\y}$ such that $\b{A} \models t(\oa,\ob) \neq \bot$, the $\mathcal{L}_A$-theory $T'' = \CTLfp \cup \{t(\oa,\ob) \neq \bot\} \cup \{\phi_j(\oa) \to \psi_j(\oa)\}$ is consistent. 
\end{claimfirst}
\begin{pfclaim}~\ref{c:consistentsubset}. 
Since $t(\oa,\ob) \neq \bot$ holds in the rooted $\CTLf$-algebra $\b{A}$, the Completeness Theorem~\ref{t:completeness} gives that there exists a tree model $(S,\sigma)$ of $t(\p,\y)$. Since $(S_\omega,\sigma_\omega)$ is bisimilar to $(S,\sigma)$ by Lemma~\ref{l:omegaunravel}, $(S_\omega,\sigma_\omega)$ is also a model of $t(\p,\y)$, i.e., $\b{P}(S_\omega), V_{\sigma_\omega} \models t(\p,\y) \neq \bot$. Moreover, by Proposition~\ref{p:concludeequiv}, $\b{P}(S_\omega), V_{\sigma_\omega} \models \phi_j(\p) \to \psi_j(\p)$, so that $(\b{P}(S_\omega), V_{\sigma_\omega})$ is a model of the theory $T''$.
\end{pfclaim}
Now, given an arbitrary finite subset $U$ of $T'$, list the finitely many terms $t_1(\oa,\ob_1),\dots,t_m(\oa,\ob_m)$ occurring in $U$. Put $\ob := \bigcup_{i=1}^m \ob_i$ and $t(\oa,\ob) := \bigwedge_{i=1}^m t_i$. By Claim~\ref{c:consistentsubset}, pick a model $\b{A}$ of $T \cup \{t(\oa,\ob) \neq \bot\} \cup \{\phi_j(\oa) \to \psi_j(\oa)\}$. Then in particular $\b{A} \models t_i(\oa,\ob_i) \neq \bot$ for each $i$, since $t(\oa,\ob) \leq t_i(\oa,\ob_i)$. Hence, $\b{A}$ is a model of $U$.
\end{proof}

\begin{reprema}{r:nomodels}
If $\b{A}$ is a model of $(\CTLfp)^*$, then the only atom of $\b{A}$ is $\I$.
\end{reprema}
\begin{proof}
Let $\phi(p,x)$ be the formula  $(x\leq p) \to [(x=\bot) \vee (x=p)]$. By Lemma~\ref{l:justequations}, convert $\phi$ into an equation $t(p,x)=\top$. Let $\psi(p)$ be the existential formula corresponding to $\forall x\,(t(p,x) = \top)$, as in~\eqref{eq:psij}. Notice that $\psi$ is equivalent to $p \leq I$, using Proposition~\ref{p:concludeequiv} and the fact that $\phi$ is equivalent to $p \leq I$ are equivalent on $\omega$-unravelled trees: the only subset $p$ of a tree $S$ which remains a singleton in the  unravelling $S_\omega$ is the singleton $\{s_0\}$, where $s_0$ is the root of $S$.
Since one of the axioms of $(\CTLfp)^*$ says that $\forall p [(\forall x\,(t(p,x) = \top)) \to \psi]$, this means that in the models of $(\CTLfp)^*$ the only atom is $\I$. 
\end{proof}

\subsection{Proofs for Subsection~\ref{ss:modelcomp-binary}}
\begin{repproposition}{p:b-auttoterm}
For any parity tree automaton $\mathcal{A}=(Q, q_I, \Delta, \Omega)$ over $\Sigma:=\mathcal{P}(\p)$ with set of states $\q$, there exists an $\mathcal{L}_{0,1}$-term $\acc_{\mathcal{A}}(\p,\q)$ such that for any $\p$-colouring $\sigma: 2^*\longrightarrow \mathcal{P}(\p)$, we have
\[ \mathcal{A} \text{ accepts } (2^*,\sigma) \iff \b{P}(2^*),V_{\sigma} \models \exists \q \; \acc_{\mathcal{A}}(\p,\q) = \top. \]
\end{repproposition}

\begin{proof}
As in the proof of Proposition~\ref{p:auttoterm}, one encodes the parity acceptance condition into a $\CTLfpb$-formula. For a triple $\theta=(\alpha,q_0, q_1)$ (with $\alpha \in \mathcal{P}(\p), q_0, q_1\in Q$), write $\bullet \theta$
for 
\[ \X_0(q_0) \wedge \X_1(q_1)\wedge \bigwedge_{p \in \alpha} p \wedge \bigwedge_{p\not\in\alpha}\neg p  ~. \]
 The required $\mathcal{L}_{0,1}$-term $\acc_{\mathcal{A}}(\p,\q)$ is taken to be $\acc_1 \wedge \acc_2 \wedge \acc_3$, where  

\begin{align*}
\acc_1(\p,\q) &:= \neg \I \vee q_I,\\
\acc_2(\p,\q) &:= \bigvee_{q \in \q}\left( 
  \begin{aligned}
    & q \wedge \bigwedge_{q' \in \q \setminus\{q\}} \neg q' \wedge 
    \\ &
    \bigvee\left\{  \bullet \theta \ | \   (q, \theta) \in \Delta\right\}
  \end{aligned}
\right),\\
\acc_3(\p,\q) &:= \bigwedge \left\{\AF\left(\bigvee_{\Omega(q')<n}q',\bigwedge_{\Omega(q)=n}\neg q\right) \right\},
\end{align*}
where the last conjunction is taken over the set of the odd numbers
 $n$ that belongs to the range of $\Omega$.
\end{proof}

}

\end{document}